\documentclass[12pt]{amsart}
\usepackage{txfonts}
\usepackage{amscd,amssymb,mathabx,subfigure}
\usepackage[all]{xy}
\usepackage{graphicx,calrsfs}
\usepackage[colorlinks,plainpages,backref,urlcolor=blue]{hyperref}
\usepackage{tikz}
\usetikzlibrary{calc}
\usepackage{mdwlist}
\usepackage{mathtools,extarrows}
\usepackage{hhline,booktabs}

\newtheorem{theorem}{Theorem}[section]
\newtheorem{corollary}[theorem]{Corollary}
\newtheorem{lemma}[theorem]{Lemma}
\newtheorem{prop}[theorem]{Proposition}

\theoremstyle{definition}
\newtheorem{definition}[theorem]{Definition}
\newtheorem{example}[theorem]{Example}
\newtheorem{remark}[theorem]{Remark}

\newcommand\Tstrut{\rule{0pt}{2.6ex}}       
\newcommand\Bstrut{\rule[-0.9ex]{0pt}{0pt}} 
\newcommand{\TBstrut}{\Tstrut\Bstrut} 

\DeclareMathAlphabet{\pazocal}{OMS}{zplm}{m}{n}

\def\dot{\mathchar"013A}
\newcommand{\hdot}{{\raise1pt\hbox to0.35em{\Huge $\dot$}}}

\definecolor{dkgreen}{RGB}{0,100,0}
\definecolor{dkbrown}{RGB}{139,69,19}

\begin{document}
\date{October 10, 2018}

\title[Strong and shifted stability]{%
Strong and shifted stability for the cohomology of configuration spaces}

\author[B.~Berceanu]{Barbu Berceanu}
\address{Simion Stoilow Institute of Mathematics,
P.O. Box 1-764, RO-014700 Bucharest, Romania,\newline
Abdus Salam School of Mathematical Sciences GCU Lahore, Pakistan }
\email{Barbu.Berceanu@imar.ro}

\author[M.~Yameen]{Muhammad Yameen}
\address{Abdus Salam School of Mathematical Sciences GCU Lahore, Pakistan}
\email{yameen99khan@gmail.com}

\thanks{$^2$ This research was partially supported by Higher Education Commission of Pakistan}

\subjclass[2010]{Primary: 55R80, 57N65, 57R19 ; Secondary: 55P62.}

\keywords{unordered configuration spaces, homological stability, F\'{e}lix-Thomas model, Knudsen model}

\begin{abstract}
Homological stability for unordered configuration spaces of connected manifolds was discovered by Th. Church and extended by O. Randal-Williams and B. Knudsen: $H_{i}(C_{k}(M);\mathbb{Q})$ is constant for $k\geq f(i)$. We characterize the manifolds satisfying strong stability: $H^{*}(C_{k}(M);\mathbb{Q})$ is constant for $k\gg 0$. We give few examples of manifolds whose top Betti numbers are stable after a shift of degree.
\end{abstract}

\maketitle
\setcounter{tocdepth}{1}
\tableofcontents

\section{Introduction and statement of results}
\label{sec:intro}


For a topological space $X$ we consider the $k$-points {\em ordered configuration space} $F_{k}(X)$ and the {\em unordered configuration space} $C_{k}(X)$ defined by
$$ F_{k}(X)=\{(x_{1},\ldots,x_{k})\in X^{k}| x_{i}\neq x_{j}\,for\,i\neq j\},\quad C_{k}(X)=F_{k}(X)/S_{k}, $$
with the induced topology and quotient topology respectively.

One of the first results in the study of configuration spaces was the cohomological strong stability theorem of V. I. Arnold \cite{A}: for $k\geq2$
$$ H^{i}(C_{k}(\mathbb{R}^{2});\mathbb{Q}) =\begin{cases}
      \mathbb{Q}, & \mbox{if }i=0, 1\\
      0, & \mbox{if }i\geq2.\\
   \end{cases}
$$
The abelianization of Artin braid group is $\mathbb{Z}$; Arnold proved that higher cohomology groups are finite groups (they are trivial for $i\geq k$) and also he proved cohomological stability for the torsion part:
$$ H^{i}(C_{2i-2}(\mathbb{R}^{2});\mathbb{Z})\cong H^{i}(C_{2i-1}(\mathbb{R}^{2});\mathbb{Z})  \cong H^{i}(C_{2i}(\mathbb{R}^{2});\mathbb{Z})\cong\ldots $$
The isomorphisms (for $k$ large depending on $i$)
$$ H^{i}(C_{k}(M);\mathbb{Q})\cong H^{i}(C_{k+1}(M);\mathbb{Q})\cong H^{i}(C_{k+2}(M);\mathbb{Q})\cong\ldots $$
were generalized for open manifolds by D. McDuff \cite{MD} and G. Segal \cite{S}. Using representation stability, Th. Church \cite{Ch} proved that
$$ H^{i}(C_{k}(M);\mathbb{Q})\cong H^{i}(C_{k+1}(M);\mathbb{Q})\cong H^{i}(C_{k+2}(M);\mathbb{Q})\cong\ldots $$
for $k>i$ and $M$ a connected oriented manifold of finite type. This result was extended by O. Randal-Williams \cite{RW1} and B. Kundsen \cite{K}.

We will define and study other stability properties of the rational cohomology of unordered configuration spaces of connected manifolds of finite type. Without a special mention, the (co)homology groups will have coefficients in $\mathbb{Q}.$ For a manifold $M$ of dimension $n$, its Betti numbers, its Poincar\'{e} polynomial and its {\em total Betti number} are defined by
$$ \beta_{i}(M)=\mbox{dim}_{\mathbb{Q}}H^{i}(M),\quad P_{M}(t)=\sum_{i=0}^{n}\beta_{i}(M)t^{i},\quad \beta(M)=P_{M}(1) . $$

The {\em top Betti number} $\beta_{\tau}(M)$ is the last non-zero Betti number of $M$, its cohomological dimension is
{$\mbox{cd}(M)=\tau$} and its $q$-{\em truncated Poincar\'{e} polynomial} contains the last $q$-Betti numbers:
$$ P_{M}^{[q]}(t)=\beta_{\tau-q+1}(M)t^{\tau-q+1}+\ldots+\beta_{\tau}(M)t^{\tau}. $$

A space $X$ has {\em even cohomology} if all its odd Betti numbers are zero, and a space $Y$ has {\em odd cohomology } if all its positive even Betti numbers are zero (and it is path connected):
$$ H^{*}(X)=H^{even}(X),\mbox{ respectively }\tilde{H}^{*}(Y)=H^{odd}(Y). $$

We say that a manifold $M^{4m}$ is a {\em homology projective plane} if its Poincar\'{e} polynomial is $1+t^{2m}+t^{4m}.$
\begin{remark} \label{rem.1}
There are classical results on topological spaces with three nonzero integral Betti numbers; see many example in the paper of J. Eells and N. Kupers "Manifolds which are like projective planes" \cite{E-K}. In all of them $m$ takes values $1,\,2,\,4.$ More rational projective planes are described in \cite{PY},   \cite{FZ}, \cite{KS} and \cite{SU}.
\end{remark}

Some algebraic models for the configuration spaces are bigraded and this will give a bigrading on the cohomology of $C_k(M)$:
$$ H^*(C_k(M))=\bigoplus_{i\geq 0}H^i(C_k(M)), \quad H^i(C_k(M))=\bigoplus_{j\geq 0}H^{i,j}(C_k(M)) $$
and we will use the two-variables Poincar\'{e} polynomial
$$ P_{C_k(M)}(t,s)=
\sum _{i,j\geq 0}\mbox{dim}_{\mathbb{Q}}H^{i,j}(C_k(M))t^is^j= 
\sum _{i,j\geq 0}\beta_{i,j}t^is^j $$
(of course we have $ P_{C_k(M)}(t)=P_{C_k(M)}(t,1)$).

We will prove a bigraded version of classical stability:
\begin{theorem}\label{stab}
For a manifold $M^{2m}$ we have:\\
a) if $i\leq k$
$$ H^{i,0}(C_{k}(M))\cong H^{i,0}(C_{k+1}(M))\cong H^{i,0}(C_{k+2}(M))\cong\ldots $$
b) if $j\geq 1$ and $i\leq k+(2m-2)j-1$
$$ H^{i,j}(C_{k}(M))\cong H^{i,j}(C_{k+1}(M))\cong H^{i,j}(C_{k+2}(M))\cong\ldots. $$
\end{theorem}
Here is our first definition:
\begin{definition} \label{def.1}
A connected manifold satisfies the {\em strong stability condition} for its unordered configuration spaces $\{C_{k}(M)\}_{k\geq 1},$ with {\em range} $r,$ if and only if the cohomology groups are eventually constant:
$$ H^{*}(C_{r}(M))\cong H^{*}(C_{r+1}(M))\cong H^{*}(C_{r+2}(M))\cong\ldots. $$
\end{definition}
In the literature there are few examples of manifolds satisfying this condition: $\mathbb{R}^{2}$-V. I. Arnold \cite{A}, $\mathbb{R}^{n}$-F. Cohen \cite{Co} (see also \cite{RW2}), $S^{2}$-M. B. Sevryuk \cite{Se}, $S^{n}$-P. Salvatore \cite{Sa} (see also \cite{RW2}), $\mathbb{C}\mbox{P}^{2}$-Y. F\'{e}lix and D. Tanr\'{e} \cite{F-Ta} (see also \cite{K-M}), $\mathbb{R}\mbox{P}^{n}$-B. Knudsen \cite{K}.

\begin{remark}  \label{rem.2}
The (strong) stability property is missing in the torsion part of homology: E. Fadell and J. Van Buskirk \cite{F-VB} computed the first homology group of $C_{k}(S^{2}):\,\,\mathbb{Z}/(2k-2)\mathbb{Z}.$ Also D. B. Fuchs \cite{F} proved that, for an arbitrary degree $i$, one can find a large $k$  such that $H^{\geq i}(C_{k}(\mathbb{R}^{2});\mathbb{Z}_{2})$ is non-zero, hence $S^{2}$ and $\mathbb{R}^2$ have not the strong stability property with integral cohomology.
\end{remark}
The first results say the previous examples are essentially all manifolds with the strong stability property.

\begin{theorem}  \label{th.1}
A manifold of odd dimension has the strong stability property if and only if $M$ has odd cohomology. In this case the range of stability is:
$$ r =\begin{cases}
      1, & \mbox{if $M$ is rationaly acyclic,}\\
      \beta(M)-1, & \mbox{otherwise}.\\
   \end{cases} $$
\end{theorem}
\begin{theorem}  \label{th.2}
A closed oriented manifold of even dimension has the strong stability property if and only if $M$ is a homology sphere or a homology projective plane and the ranges of stability are $3$ and $4$ respectively.
\end{theorem}
\begin{corollary} \label{cor.1}
A closed oriented manifold $M$ has the strong stability property if and only if $M$ is a homology sphere or a homology projective plane.
\end{corollary}
Various results and conjectures on stability of the top Betti number could be found in the literature: J. Miller and J. Wilson \cite{MW}, Th. Church, B. Farb and A. Putman \cite{CFP} or M. Maguire \cite{M} and, recently, S. Galatius, A. Kupers and O. R. Williams \cite{GKW}. Here is our second definition:
\begin{definition}
A connected manifold $M$ satisfies the {\em shifted stability condition} for its unordered configuration spaces 
$\{C_{k}(M)\}_{k\geq 1}$, with {\em range} $r$, {\em shift} 
$\sigma$ and {\em length} $q$  $(r,\sigma,q\geq1)$, if and only if the $q$-truncated Poincar\'{e} polynomial is stable after a shift: for any $k\geq r$ we have
$$ P^{[q]}_{C_{k+1}(M)}(t)=t^{\sigma}P^{[q]}_{C_{k}(M)}(t). $$
\end{definition}
We give two examples, $\mathbb{C}\mbox{P}^{1}\times\mathbb{C}\mbox{P}^{1}$ and $\mathbb{C}\mbox{P}^{3}$,  where classical stability and shifted stability properties combined give the entire two variable Poincar\'{e} polynomials:
\begin{prop}\label{prop.5}
The product of two projective lines, $\mathbb{C}\emph{P}^{1}\times\mathbb{C}\emph{P}^{1},$ has the shifted stability property with range 8, shift 2 and length 5:
$$ P^{[5]}_{C_{k+1}(\mathbb{C}\emph{P}^{1}\times\mathbb{C}\emph{P}^{1})}(t,s)=t^{2}P^{[5]}_{C_{k}(\mathbb{C}\emph{P}^{1}\times\mathbb{C}\emph{P}^{1})}(t,s)\quad \mbox{for}\,\,k\geq8.$$
More precisely, for $k\geq8$, we have:
\begin{align*}
P_{C_{k}(\mathbb{C}\emph{P}^{1}\times\mathbb{C}\emph{P}^{1})}(t,s)&=1+2t^{2}+3t^{4}+2t^{6}+2t^{8}+\ldots+2t^{2k}+\\
&+s(2t^{7}+4t^{9}+5t^{11}+4t^{13}+4t^{15}+\ldots+4t^{2k+1}+2t^{2k+3})+\\
&+s^{2}(t^{14}+2t^{16}+2t^{18}+\ldots+2t^{2k+4}).
\end{align*}
\end{prop}
\begin{prop}\label{prop.6}
The complex projective space, $\mathbb{C}\emph{P}^{3}$, has the shifted stability property with range 8, shift 2 and length 6:
$$ P^{[6]}_{C_{k+1}(\mathbb{C}\emph{P}^{3})}(t,s)=t^{2}P^{[6]}_{C_{k}(\mathbb{C}\emph{P}^{3})}(t,s)\quad \mbox{for}\,\,k\geq8. $$
More precisely, for $k\geq8$, we have:
\begin{align*}
 P_{C_{k+1}(\mathbb{C}\emph{P}^{3})}(t,s)&=1+t^{2}+2t^{4}+2t^{6}+2t^{8}+t^{10}+t^{12}+\ldots+t^{2k}+\\
 &+s(t^{11}+2t^{13}+3t^{15}+3t^{17}+3t^{19}+2t^{21}+2t^{23}+\ldots+2t^{2k+5}+t^{2k+7})+\\
 &+s^{2}(t^{24}+t^{26}+\ldots+t^{2k+12}).
\end{align*}
\end{prop}
More examples will be given in \cite{BY}.

The computation of $H^{*}(C_{k}(M)),$ using cohomology algebra of $M,$ is easy in the odd dimensional case (see \cite{B-C-T} and \cite{F-Ta}):
\begin{theorem} \label{th.4} (C-F.~B\"{o}digheimer, F.~Cohen, L.~Taylor -- Y.~F\'{e}lix, D.~Tanr\'{e})

For a manifold $M^{2m+1}$ we have
$$ H^*(C_k(M))=Sym^k(H^*(M)). $$ 
\end{theorem}

In the even dimensional case, the cohomology groups 
$H^{*}(C_{*}(M))$ are given by the cohomology of a differential bigraded algebra $(\Omega^{*}(*)(V^{*},W^{*}),\partial)$ introduced by Y. F\'{e}lix and J. C. Thomas \cite{F-Th} and extended by B. Kundsen \cite{K} (the two graded vector spaces $V^{*},$ and $W^{*}$ and the differential 
$\partial$ depend on various cohomology groups of $M$ and cohomology product):
\begin{theorem}\label{th.5}  (Y.~F\'{e}lix, J.~C.~Thomas -- B.~Knudsen)

For a manifold $M^{2m}$ we have
$$ H^{*}(C_{k}(M))\cong H^{*}(\Omega^{*}(k)(V^{*},W^{*}),\partial). $$
\end{theorem}
We recall the definition of $V^{*},\,W^{*}$ and $\partial$ in Section 4, for a closed oriented manifold $M^{2m},$ and in Section 5, for an arbitrary even dimensional manifold. In Section 2 we introduce the algebraic tool to analyze F\'{e}lix-Thomas model and Knudsen model, a sequence of weighted spectral sequences. As a first application we give the proof of Theorem \ref{stab} and an improved version of it. The proof of Theorem \ref{th.1} is given in Section 3 and the proof of Theorem \ref{th.2} in Section 5. Partial results for even dimensional manifolds, open or non-orientable, are presented in Section 5. In Section 6 we introduce three new notions of shifted stability and we describe their relations. Two necessary conditions for these shifted stability conditions are given. Section 7 contains stability properties of $\mathbb{C}\mbox{P}^{1}\times\mathbb{C}\mbox{P}^{1}$ and $\mathbb{C}\mbox{P}^{3}$ and the proofs of Propositions \ref{prop.5} and \ref{prop.6}.


\section{Weighted spectral sequences}

In this section we analyze algebraic properties of the differential bigraded algebra $(\Omega^{*}(*)(V^{*},W^{*}),\partial)$ introduced by Y. F\'{e}lix and J. C. Thomas \cite{F-Th} and extended by B. Knudsen \cite{K}.

Let us introduced some notation. For a graded $\mathbb{Q}$-vector space $A^{*}=\oplus_{i\in\mathbb{Z}}A^{i}$ we will use the notation
$$ A^{\geq q}=\bigoplus_{i\geq q}A^{i},\quad A^{even}=\bigoplus_{i\in\mathbb{Z}}A^{2i},\quad \tilde{A}^{*}=\bigoplus_{i\neq0}A^{i}, $$
and similarly $A^{\leq q}$ and $A^{odd}$; the degree $i$ component of the shifted graded space $A^{*}[r]$ is $A^{i+r}.$ We suppose that $A^{*}$ is connected: if $A^{0}\neq0,$ then $A^{0}\cong\mathbb{Q}$. The symmetric algebra $Sym(A^{*})$ is the tensor product of a polynomial algebra and an exterior algebra:
$$ Sym(A^{*})=\bigoplus_{k\geq0}Sym^{k}(A^{*})=Polynomial(A^{even})\bigotimes Exterior(A^{odd}), $$
where $Sym^{k}$ is generated by the monomials of length $k$ (without any other convention, the elements in $A^{*}$ have length 1).

Fix a positive even number $2m$, the ``geometric " dimension, and consider two graded vector spaces $V^{*},\,W^{*}$, and a degree 1 linear map $\partial_{W}:$
$$ V^{*}=\bigoplus_{i=0}^{2m}V^{i},\,\,W^{*}=\bigoplus_{j=2m-1}^{4m-1}W^{j},\,\, \partial_{W}:\,W^{*}\longrightarrow Sym^{2}V^{*}.  $$
By definition, the elements in $V^{*}$ have length 1 and weight 0 and the elements in $W^{*}$ have length 2 and weight 1. We choose bases in $V^{i}$ and $W^{j}$ as
$$  V^{i}=\mathbb{Q}\langle v_{i,1},v_{i,2},\ldots\rangle,\quad W^{j}=\mathbb{Q}\langle w_{j,1},w_{j,2},\ldots\rangle $$
(the degree of an element is marked by the first lower index; $x_{i}^{q}$ stands for the product $x_{i}\wedge x_{i}\wedge\ldots\wedge x_{i}$ of $q$-factors). Always we take $V^{0}=\mathbb{Q}\langle v_{0}\rangle$. The graded vector space $V^{*}$ is $(h-1)$-connected if $V^{*}=V^{0}\oplus V^{\geq h}$.

The definition of the bigraded differential algebra $\Omega^{*}(k)$ is
$$ \Omega^{*}(*)(V^{*},W^{*})=\bigoplus_{k\geq1}\Omega^{*}(k)(V^{*},W^{*}), $$
$$ \Omega^{*}(k)(V^{*},W^{*})=\bigoplus_{i\geq 0}
\Omega^{i}(k)(V^{*},W^{*})=Sym^{k}(V^{*}\oplus W^{*}), $$
where the total degree $i$ is given by the grading of $V^{*}$ and $W^{*}$ and the length degree $k$ is the multiplicative extension of length on $V^{*}$ and $W^{*}$. The differential is defined by $\partial|_{V^{*}}=0,\,\partial|_{W^{*}}=\partial_{W}$ and it has bidegree $(1,0)$. For instance,
$$ H^{*}(\Omega^{*}(1)(V^{*},W^{*}),\partial)=H^{*}(Sym^{1}(V^{*}),\partial=0)=V^{*}. $$

We are interested in the stability properties of the sequence of cohomology spaces $\{H^{*}(\Omega^{*}(k)(V^{*},W^{*}),\partial)\}_{k\geq1}$ i.e. we have to compare $H^{*}(\Omega^{*}(k-1)(V^{*},W^{*}),\partial)$ with $H^{*}(\Omega^{*}(k)(V^{*},W^{*}),\partial)$, and for this we introduce a sequence of weighted spectral sequences.

The subspace of $\Omega^{*}(k)$ containing the elements of weight 
$\omega$ is denoted ${}^{\omega}\Omega^{*}(k)$ and we have
$$  \Omega^{*}(k)(V^{*},W^{*})=\bigoplus_{\omega=0}^{\lfloor\frac{k}{2}\rfloor}{}^{\omega}\Omega^{*}(k),\quad{}^{0}\Omega(k)=Sym^{k}(V^{*}),  $$
$$ \partial:{}^{\omega}\Omega^{*}(k)\longrightarrow{}^{\omega-1}\Omega^{*+1}(k). $$
We define an increasing filtration of subcomplexes
$\{F^{i}\Omega^{*}(k)(V^{*},W^{*})\}_{i=0,\ldots,2m}$:
$$ F^{i}\Omega^{*}(k) =[V^{\leq i}\otimes \Omega^{*}(k-1(V^{*},W^{*})]+[W^{\leq 2i}\otimes \Omega^{*}(k-2)(V^{*},W^{*})].  $$
Obviously we have
$$ \begin{array}{rcl}
     \partial (V^{\leq i}\otimes \Omega^*(k-1))  & \subset & V^{\leq i}\otimes
                                                  \Omega ^*(k-1) \mbox{ and} \\
     \partial(W^{\leq 2i}\otimes \Omega^*(k-2))  & \subset & V^{\leq i}\otimes
                              \Omega^*(k-1)+W^{\leq 2i}\otimes\Omega^*(k-2).
\end{array} $$
The filtration $\{F^{i}\}_{i=0,\ldots,2m}$ and the weight decomposition $\{{}^{\omega}\Omega^{*}(k)\}_{\omega=0,\ldots,\lfloor\frac{k}{2}\rfloor}$ are compatible:
$$ F^i\Omega^{*}(k)=F^i\cap {}^0\Omega^*(k)\oplus F^i\cap{}^1\Omega^*(k)\oplus\ldots
      \oplus F^i\cap{}^{\lfloor\frac{k}{2}\rfloor}\Omega^*(k)=\bigoplus_{\omega=0}
      ^{\lfloor\frac{k}{2}\rfloor}F^i\Omega^*(k), $$
hence the spectral sequence $E^{*,*}_{*}(k)$ associated with the filtration $\{F^{i}\Omega^{*}(k)\}_{i=0,\ldots,2m}$ is weight-splitted at any page:
$$ E^{*,*}_*(k)=\bigoplus_{\omega=0}^{\lfloor\frac{k}{2}\rfloor}
         {}^{\omega}E^{*,*}_*(k), $$
with differential
$$ d_{r}^{i,j}:\,{}^{\omega}E^{i,j}_{r}(k)\longrightarrow{}^{\omega-1}E^{i-r,j+r+1}_{r}(k). $$

Some general properties of these spectral sequences are obvious:
\begin{prop} \label{prop.1}
Every $E_*^{*,*}(k)$ is a first quadrant spectral sequence; as
$E_r^{\geq 2m+1,q}(k)=0$, the spectral sequence degenerate at $2m+1$.
\end{prop}
Here are few pictures of the polygons containing the support of the weighted components of the first page of the spectral sequences $E_*^{*,*}(k)$:
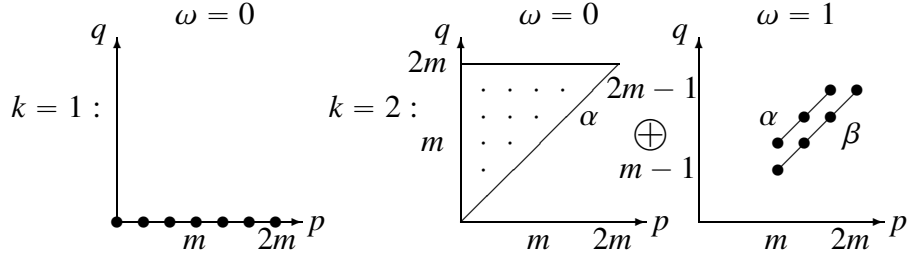
\begin{center}
\begin{picture}(360,110)
\multiput(170,20)(90,0){2}{\vector(0,1){70}}   
\put(40,20){\vector(0,1){70}}                 
\multiput(170,20)(90,0){2}{\vector(1,0){70}}   
\put(40,20){\vector(1,0){70}}            \put(230,37){$m-1$}
\multiput(37,17)(10,0){7}{$\bullet$}     
\put(170,80){\line(1,0){60}}  
\multiput(287,37)(10,10){4}{$\bullet$}   \put(225,67){$2m-1$}
\multiput(287,47)(10,10){3}{$\bullet$}   
\put(170,20){\line(1,1){60}}
\multiput(195,10)(90,0){2}{$m$}          \put(30,90){$q$}
\multiput(220,10)(90,0){2}{$2m$}         \put(65,10){$m$}       
\multiput(62,95)(130,0){2}{$\omega=0$}   
\put(282,95){$\omega=1$}
\put(235,50){$\bigoplus $}               
\put(290,40){\line(1,1){30}}
\multiput(215,57)(68,0){2}{$\alpha$}     
\put(290,50){\line(1,1){20}}
\put(315,50){$\beta$} \put(0,60){$k=1:$} \put(120,60){$k=2:$} 
\multiput(160,90)(90,0){2}{$q$}          \put(112,17){$p$}
\multiput(242,17)(90,0){2}{$p$}          \put(155,47){$m$} 
\multiput(177,37)(0,10){4}{$\cdot$}      \put(207,67){$\cdot$} 
\multiput(187,47)(0,10){3}{$\cdot$}      \put(150,77){$2m$}
\multiput(197,57)(0,10){2}{$\cdot$}      \put(93,10){$2m$}
\end{picture}
\end{center}

The equations of the lines are $\alpha:\,q=p$ and $\beta:\,q=p-1$.
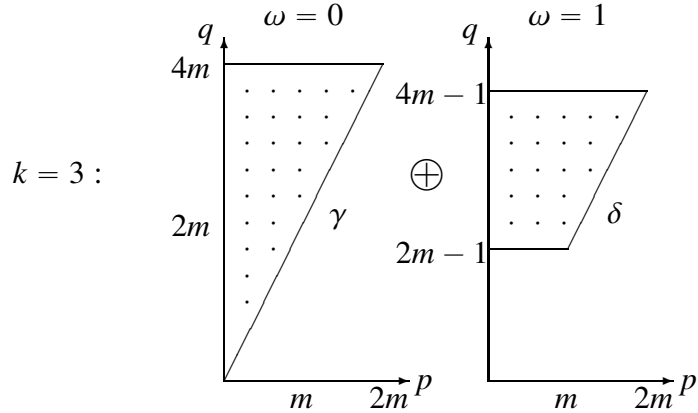
\begin{center}
\begin{picture}(360,170)
\multiput(110,20)(100,0){2}{\vector(0,1){130}} 
\put(110,140){\line(1,0){60}}        \put(90,135){$4m$} 
\multiput(110,20)(100,0){2}{\vector(1,0){70}}  
\put(110,20){\line(1,2){60}}         \put(175,65){$2m-1$}
\multiput(117,47)(0,10){9}{$\cdot$}  \put(90,75){$2m$} 
\multiput(127,67)(0,10){7}{$\cdot$}  \put(255,80){$\delta$}
\multiput(137,87)(0,10){5}{$\cdot$}  \put(175,125){$4m-1$}  
\put(210,70){\line(1,0){30}}
\multiput(135,10)(100,0){2}{$m$}     
\multiput(165,10)(100,0){2}{$2m$}   \put(240,70){\line(1,2){30}}
\put(125,155){$\omega=0$}           \put(225,155){$\omega=1$}
\put(180,95){$\bigoplus $}          
\put(210,130){\line(1,0){60}}
\multiput(100,150)(100,0){2}{$q$}    \put(30,95){$k=3:$}         
\multiput(182,17)(100,0){2}{$p$}     \put(150,80){$\gamma$} 
\multiput(147,107)(0,10){3}{$\cdot$} \put(157,127){$\cdot$}
\multiput(217,77)(10,0){3}{\multiput(0,0)(0,10){5}{$\cdot$}}
\multiput(247,97)(0,10){3}{$\cdot$}  \put(257,117){$\cdot$}
\end{picture}
\end{center}
The equations of the lines are $\gamma:\,q=2p$ and $\delta:\,q=p-1$.
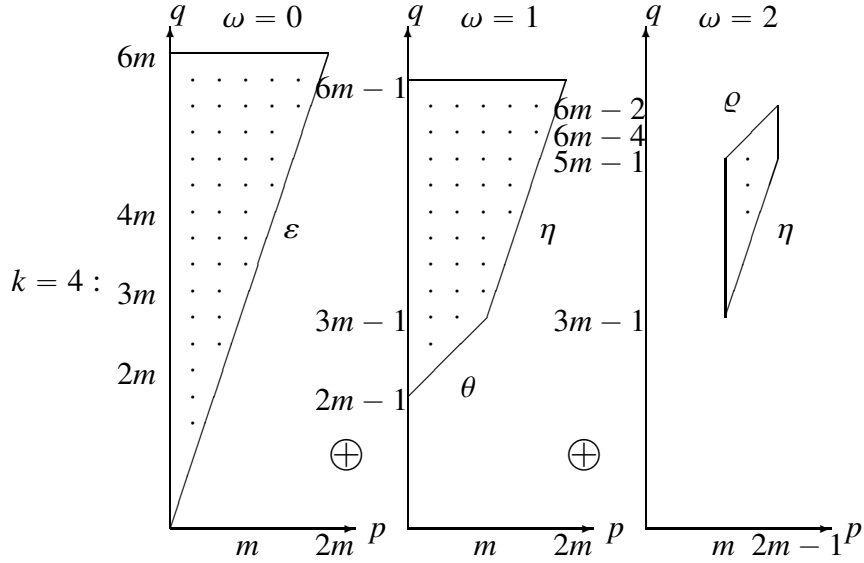
\begin{center}
\begin{picture}(360,220)
\multiput(60,20)(90,0){3}{\vector(0,1){190}}
\multiput(60,20)(90,0){3}{\vector(1,0){70}}
\put(60,20){\line(1,3){60}}             
\put(60,200){\line(1,0){60}}
\multiput(67,57)(0,10){14}{$\cdot$}     \put(40,75){$2m$}
\multiput(77,87)(0,10){11}{$\cdot$}     \put(40,105){$3m$}
\multiput(87,117)(0,10){8}{$\cdot$}     \put(40,135){$4m$}
\multiput(85,10)(90,0){3}{$m$}          
\put(150,190){\line(1,0){60}}
\multiput(115,10)(90,0){2}{$2m$}        \put(40,195){$6m$}
\put(115,65){$2m-1$}\put(205,95){$3m-1$}\put(115,95){$3m-1$}
\put(115,183){$6m-1$}                   \put(205,165){$6m-4$}
\multiput(120,45)(90,0){2}{$\bigoplus $}\put(80,210){$\omega=0$}
\put(103,130){$\varepsilon$}            \put(0,110){$k=4:$}\multiput(200,130)(90,0){2}{$\eta$}     \put(170,70){$\theta$}
\put(270,100){\line(0,1){60}}     \put(270,100){\line(1,3){20}}
\multiput(60,213)(90,0){3}{$q$}   \put(270,160){\line(1,1){20}}
\multiput(135,17)(90,0){3}{$p$}   \put(290,160){\line(0,1){20}}
\multiput(97,147)(0,10){5}{$\cdot$} \put(150,70){\line(1,1){30}}
\multiput(107,177)(0,10){2}{$\cdot$}    \put(205,155){$5m-1$}
\multiput(157,87)(0,10){10}{$\cdot$}    \put(280,10){$2m-1$}
\multiput(167,97)(0,10){9}{$\cdot$}     
\put(170,210){$\omega=1$}
\multiput(177,107)(0,10){8}{$\cdot$}
\put(180,100){\line(1,3){30}}
\multiput(187,137)(0,10){5}{$\cdot$}    \put(205,175){$6m-2$}
\multiput(197,167)(0,10){2}{$\cdot$}    
\put(260,210){$\omega=2$}
\multiput(277,137)(0,10){3}{$\cdot$}    \put(270,180){$\varrho$}
\end{picture}
\end{center}

The equations of the lines are $\varepsilon:\,q=3p$, $\theta:\,q=p+2m-1$, $\eta:\,q=3p-1$ and $\varrho:\,q=p+4m-1$.
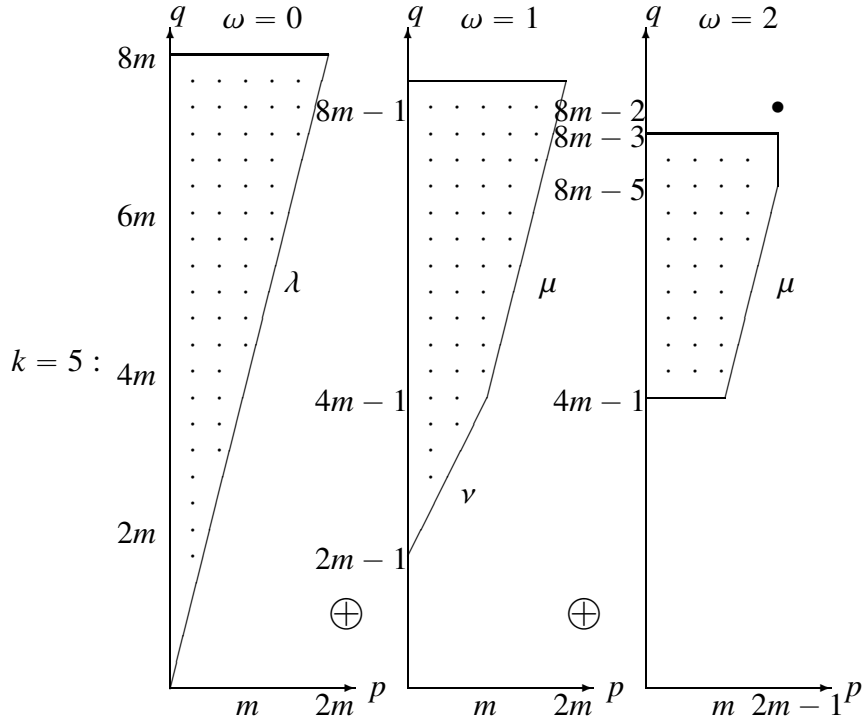
\begin{center}
\begin{picture}(360,280)
\multiput(60,20)(90,0){3}{\vector(0,1){250}}
\put(60,260){\line(1,0){60}}         \put(60,20){\line(1,4){60}}
\multiput(60,20)(90,0){3}{\vector(1,0){70}}
\multiput(67,67)(0,10){19}{$\cdot$}  \put(280,10){$2m-1$}
\multiput(77,107)(0,10){15}{$\cdot$}  
\put(150,250){\line(1,0){60}}
\multiput(87,147)(0,10){11}{$\cdot$}     
\put(180,130){\line(1,4){30}}
\multiput(85,10)(90,0){3}{$m$}         \put(205,225){$8m-3$}
\multiput(115,10)(90,0){2}{$2m$}       \put(205,235){$8m-2$}
\put(40,75){$2m$}   \put(40,255){$8m$} \put(40,135){$4m$}
\put(40,195){$6m$}                     \put(115,125){$4m-1$}
\put(115,235){$8m-1$}                  \put(205,205){$8m-5$}
\multiput(120,45)(90,0){2}{$\bigoplus $} 
\put(103,170){$\lambda$}
\multiput(200,170)(90,0){2}{$\mu$}     \put(287,237){$\bullet$}
\multiput(60,273)(90,0){3}{$q$}          
\put(240,130){\line(1,0){30}}
\multiput(135,17)(90,0){3}{$p$}  \put(270,130){\line(1,4){20}}
\put(240,230){\line(1,0){50}}    \put(290,210){\line(0,1){20}}
\multiput(97,187)(0,10){7}{$\cdot$}      \put(170,90){$\nu$}
\multiput(107,227)(0,10){3}{$\cdot$}     \put(0,140){$k=5:$}
\multiput(157,97)(0,10){15}{$\cdot$}     
\put(260,270){$\omega=2$}
\multiput(167,117)(0,10){13}{$\cdot$} \put(170,270){$\omega=1$}
\multiput(177,137)(0,10){11}{$\cdot$} \put(80,270){$\omega=0$}
\multiput(187,177)(0,10){7}{$\cdot$}  
\put(150,70){\line(1,2){30}}
\multiput(197,217)(0,10){3}{$\cdot$}     \put(205,125){$4m-1$}
\multiput(247,137)(10,0){3}{\multiput(0,0)(0,10){9}{$\cdot$}}
\multiput(277,187)(0,10){4}{$\cdot$}     \put(115,65){$2m-1$}
\end{picture}
\end{center}
The equations of the lines are $\lambda:\,q=4p$, $\mu:\,q=4p-1$ and $\nu:\,q=2p+2m-1$. Using the definition of the filtration 
$F^i$, one can describe the support of 
${}^{\omega}E_*^{*,*}(k)$, in general:
\begin{prop}  \label{prop.2}
a) If $2\omega>k$, then  ${}^{\omega}E_0^{*,*}(k)=0$.

b) The support of the weighted components of  ${}^{\omega}E_0^{*,*}(k)$ are
contained in the following regions:
$$  \begin{array}{ll}
\omega=0:     & \mbox{the triangle defined by }0\leq (k-1)p\leq q\leq 2(k-1)m;              \\
\omega=1:     & \mbox{if }k=2, \mbox{ the trapezoid defined by }m-1\leq p-1\leq q\leq
                      \min(p,2m-1);                                                         \\
              & \mbox{if }k\geq 3, \mbox{ the quaddrilateral defined by }                        \\
              & \quad \max((k-3)p+2m-1,(k-1)p-1)\leq q\leq 2(k-1)m-1;                       \\
\omega\geq 2: & \mbox{if }k=2\omega, \mbox{ the trapezoid defined by } m\leq p\leq 2m-1 \mbox{ and} \\
              & \quad (k-1)p-1\leq q\leq p+(2k-4)m-k+3;                   \\
              & \mbox{if }k\geq 2\omega +1, \mbox{ the pentagon defined by }               \\
              & \quad\max((k-2\omega-1)p+2\omega m-1,(k-1)p-1)\leq q\leq 2(k-1)m-2\omega +1 \\
              & \quad\mbox{and the exterior point }(p,q)=(2m-1,2(k-1)m-2).
\end{array} $$
\end{prop}
\begin{proof}

In the table there is a list of elements of minimal degree (in bottom position) and elements of maximal degree (in top position) in the column $F^p/F^{p-1}$ of the spectral sequence 
${}^{\omega}E_0^{*,*}(k)$:
\begin{center}
Table 1
\vspace{1pt}\\
\begin{tabular}{|c|c|c|c|}
\hline
$(\omega,k)$       & $0\leq p\leq m-1 $ & $m\leq p\leq 2m-1$ & 
                                         $p=2m$       \\
\hline
$\omega=0 $        & $ \begin{array}{l} v_pv_{4m}^{k-1}\Tstrut \\[0.1em] v_p^k     \end{array} $
                   & $ \begin{array}{l} v_pv_{4m}^{k-1}\Tstrut \\[0.1em] v_p^k     \end{array} $
                   & $ \begin{array}{c}  v_{4m}^k    \Tstrut\\[0.1em] v_{2m}^k     \end{array} $\Bstrut \\
\hline
$ \begin{array}{c} \omega=1 \\   k=2                                 \end{array} $
                   &     ---
                   & $ \begin{array}{l} w_{2p}         \\ w_{2p-1}  \end{array} $
                   & $ \begin{array}{l}  w_{4p-1}    \\  w_{4p-1}    \end{array} $\Bstrut      \\
\hline
$ \begin{array}{c} \omega=1 \\ k>2                                   \end{array} $
                   & $ \begin{array}{l} v_pv_{2m}^{k-3}w_{4m-1}
                                            \Tstrut\\[0.1em] v_p^{k-2}w_{2m-1}     \end{array} $
                   & $ \begin{array}{l}  v_pv_{2m}^{k-3}w_{4m-1}
                                           \Tstrut\\[0.1em]v_p^{k-2}w_{2p-1}       \end{array} $
                   & $ \begin{array}{l}  v_{2m}^{k-2}w_{4m-1}
                                    \Tstrut\\[0.1em] v_{2m}^{k-2}w_{4m-1}          \end{array} $\Bstrut     \\
\hline
$ \begin{array}{c} \omega\geq 2 \\  k=2\omega                        \end{array} $
                   &   ---
                   & $ \begin{array}{l} w_{2p}w_{4m-2}^{\omega-2}w_{4m-1}
                                   \Tstrut\\[0.1em]  w_{2p-1}w_{2p}^{\omega-1}     \end{array} $
                   &    ---                                                            \Bstrut\\
\hline
$ \begin{array}{c}\omega\geq 2 \\  k>2\omega                  \end{array} $
                   &  $ \begin{array}{l} v_pv_{2m}^{k-2\omega -1}w_{4m-2}^{\omega -1}w_{4m-1} \Tstrut\\[0.1em]
                        v_p^{k-2\omega}w_{2m-1}w_{2m}^{\omega -1}     \end{array} $
                   &  $ \begin{array}{c} v_pv_{2m}^{k-2\omega -1}w_{4m-2}^{\omega -1}w_{4m-1} \Tstrut\\[0.1em]
                         v_p^{k-2\omega}w_{2p-1}w_{2p}^{\omega-1}     \end{array} $
                   &     ---                       \Bstrut\\
\hline
\end{tabular}
\end{center}

There is a unique exception: if $p=2m-1$, $\omega\geq 2$ and $k\geq 2\omega +1$, the element of maximal degree is 
$ v_{2m}^{k-2\omega}w_{4m-2}^{\omega-1}w_{4m-1} $.
\end{proof}
{\em Proof of Theorem \ref{stab}.} Let us define
$$ k(j) =\begin{cases}
      k, & \mbox{if $j=0$}\\
      k+(2m-2)j-1 & \mbox{if $j\geq 1.$}\\
   \end{cases} $$
On the 0-th page of the spectral sequence ${}^{*}E^{*,*}_{*}(k+1)$ we find that ${}^{j}E^{\geq1,*}_{0}(k+1)$ has no element under the line $p+q=k+1$ for $j=0$ and nothing under the line $p+q=k+(2m-2)j$ for 
$j\geq1$.
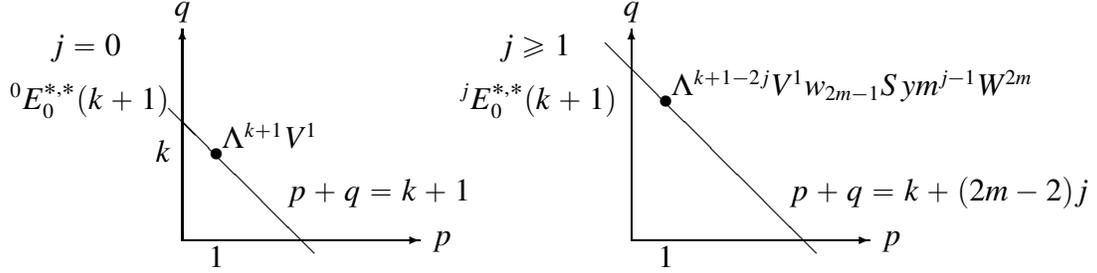
\begin{center}
\begin{picture}(420,120)
\multiput(70,20)(170,0){2}{\vector(0,1){80}} 
\multiput(67,105)(170,0){2}{$q$}       \put(310,15){\line(-1,1){80}}
\multiput(165,18)(170,0){2}{$p$}       \put(120,15){\line(-1,1){55}}
\multiput(70,20)(170,0){2}{\vector(1,0){90}}
\put(80,50){$\bullet$}                 \put(250,70){$\bullet$}
\put(5,70){${}^{0}E^{*,*}_{0}(k+1)$}   \put(190,90){$j\geq 1$}
\put(175,70){${}^{j}E^{*,*}_{0}(k+1)$} \put(20,90){$j=0$} 
\put(80,10){1}                         \put(250,10){$1$}
\put(85,55){$\Lambda^{k+1}V^{1}$}      \put(60,50){$k$}
\put(255,75){$\Lambda^{k+1-2j}V^{1}w_{2m-1}Sym^{j-1}W^{2m}$}
\put(110,35){$p+q=k+1$}                \put(300,35){$p+q=k+(2m-2)j$}
\end{picture}
\end{center}

On the column 0 we have ${}^{j}E^{0,*}_{1}(k+1)=H^{*,j}(C_{*}(M))$ and also
$$ H^{\leq k(j),j}(C_{k}(M))\cong {}^{j}E^{0,\leq k(j)}_{1}(k+1)\cong {}^{j}E^{0,\leq k(j)}_{\infty}(k+1)\cong H^{\leq k(j),j}(C_{k+1}(M)).$$ $\hfill \square$
\begin{theorem}\label{hstab}
For a $(h-1)$-connected closed orientable manifold $M^{2m}$ we have:\\
a) if $i\leq h(k+1)-1$, then
$$ H^{i,0}(C_{k}(M))\cong H^{i,0}(C_{k+1}(M))\cong H^{i,0}(C_{k+2}(M))\cong\ldots $$
b) if $j\geq1$ and $i\leq hk+(2m-h-1)j-1$, then
$$ H^{i,j}(C_{k}(M))\cong H^{i,j}(C_{k+1}(M))\cong H^{i,j}(C_{k+2}(M))\cong\ldots. $$
\end{theorem}
\begin{proof}
In this case the two graded spaces $V^{*}$ and $W^{*}$ are given by
$$ V^{*}=V^{0}\oplus V^{h}\oplus V^{h+1}\oplus \ldots V^{2m-h}\oplus V^{2m}, $$
$$ W^{*}=W^{2m-1}\oplus W^{2m+h-1}\oplus W^{2m+h}\oplus \ldots \oplus W^{4m-h-1}\oplus W^{4m-1}, $$
where the first (and last) components are one dimensional:
$V^{0}=\langle v_{0}\rangle$, $W^{2m-1}=\langle w_{2m-1}\rangle$ (see \cite{F-Th} or Section 4). As in the previous proof we find out the lowest lines:
$$ p+q=h(k+1)\mbox{ for }j=0 \mbox{ and } p+q=hk+(2m-h-1)j\mbox{ for }j\geq1. $$
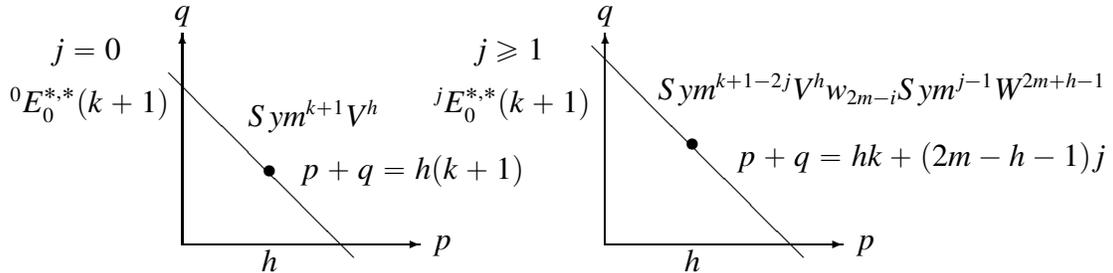
\begin{center}
\begin{picture}(420,120)
\multiput(70,20)(160,0){2}{\vector(0,1){80}} 
\multiput(67,105)(160,0){2}{$q$}     \multiput(165,18)(160,0){2}{$p$}
\multiput(70,20)(160,0){2}{\vector(1,0){90}}
\put(135,15){\line(-1,1){70}}    \put(305,15){\line(-1,1){80}}
\put(100,45){$\bullet$}              \put(260,55){$\bullet$}
\put(5,70){${}^{0}E^{*,*}_{0}(k+1)$} \put(115,45){$p+q=h(k+1)$}  
\put(165,70){${}^{j}E^{*,*}_{0}(k+1)$}
\put(20,90){$j=0$}                   \put(180,90){$j\geq 1$}
\put(95,65){$Sym^{k+1}V^{h}$}        \put(100,10){$h$} 
\put(250,75){$Sym^{k+1-2j}V^{h}w_{2m-i}Sym^{j-1}W^{2m+h-1}$}
\put(280,50){$p+q=hk+(2m-h-1)j$}     \put(260,10){$h$}
\end{picture}
\end{center}
\end{proof}
\begin{corollary}\label{htot} (Th.~Church)

For a $(h-1)$-connected closed oriented manifold $M^{2m}$ we have
$$ H^{i}(C_{k}(M))\cong H^{i}(C_{k+1}(M))\cong H^{i}(C_{k+2}(M))\cong \ldots $$
for $i\leq hk+h-2$.
\end{corollary}
\begin{proof}
If $M^{2m}$ is not a homology sphere, we have the relation $m\geq h$ and the Theorem \ref{hstab} gives the inequality $\mbox{min}\{h(k+1)-1,hk+(2m-h-1)j-1\}_{j\geq1}\geq hk+h-2$.
\end{proof}


\section{Strong stability: odd dimensional manifolds}

In this section the manifold $M$ has odd dimension.\\
{\em Proof of Theorem \ref{th.1}.} If there is a non-zero cohomology class (of positive degree) $x\in H^{2i}(M),$ then $x\wedge x\wedge\ldots\wedge x=x^{k}$ will give a non-zero cohomology class in $H^{2ki}(C_{k}(M)),$ with arbitrary high degree, hence $H^{*}(C_{k}(M))$ cannot be stable.

If $M$ has odd cohomology, with total Betti number $\beta(M)=\beta$, and a basis $\{1=x_{0},x_{1,1},x_{1,2},\ldots \}$ of $H^{*}(M)$, then the highest degree of a product of length $\beta+q-1$ is $\sum_{i=0}^{\beta_{\tau}}i\beta_{i}(M)$, the degree of the product $x_{0}^{q}\wedge(\wedge_{i\geq 0}x_{2i+1,a})$. We have the sequence of isomorphisms:
\begin{center}
\begin{picture}(360,70)
\put(0,10){$Sym^{\beta-1}(H^*(M))$}           
\put(116,10){$Sym^{\beta}(H^*(M))$}  
\put(223,10){$Sym^{\beta+1}(H^*(M))$} 
\put(10,60){$H^*(C_{\beta-1}(M))$}   \put(120,60){$H^*(C_{\beta}(M))$}             \put(230,60){$H^*(C_{\beta+1}(M))$}
\multiput(40,50)(110,0){3}{\vector(0,-1){20}} 
\multiput(45,40)(110,0){3}{$\cong$}
\multiput(80,15)(110,0){3}{\vector(1,0){30}}  
\multiput(85,20)(110,0){3}{$x_0\wedge $}
\multiput(90,5)(110,0){3}{$\cong$}   \put(340,10){$\dots$}
\end{picture}
\end{center}
$\hfill \square$\\
{\em Proof of Corollary \ref{cor.1}.} If $M^{2m+1}$ is a closed oriented manifold, by Poincar\'{e} duality we find that $\beta_{2i+1}(M)\neq0$ implies $\beta_{2m-2i}(M)\neq0;$ if $M$ has the strong stability property, this implies $m=i$.

If $M$ has even dimension, the statement is a direct consequence of Theorem \ref{th.2}.  $\hfill \square$


\section{Strong stability: closed orientable even dimensional manifolds}

First we give a necessary condition for the strong stability property.
\begin{prop}
If $M^{2m}$ has negative Euler-Poincar\'{e} characteristics, then $M^{2m}$ cannot have the strong stability property.
\end{prop}
\begin{proof}
From \cite{F-Th} and \cite{Fa} we have
$$ 1+\sum_{k=1}^{\infty}\chi(C_{k}(M))t^{k}=(1+t)^{\chi(M)}, $$
hence, if $\chi(M)$ is negative, the sequence $\{\chi(C_{k}(M))\}_{\geq1}$ is not eventually constant.
\end{proof}
In this section we analyze the strong stability property for a closed oriented manifold of even dimension $M^{2m}$.

The DG-algebra introduced by Y. F\'{e}lix and J. C. Thomas \cite{F-Th} is defined by
$$ V^{*}=H_{*}(M),\quad W^{*}=H_{*}(M)[2m-1] $$
and the differential $\partial$ is dual to the cup product
$$ H^{*}(M)\bigotimes H^{*}(M)\xrightarrow{\cup}H^{*}(M). $$
\begin{lemma}  \label{lem.1}
If $M^{2m}$ is a homology sphere, then $M$ has the strong stability property with the range of stability 3.
\end{lemma}
\begin{proof}
As $H^*(M)=\mathbb{Q}[x]/(x^2)$, the two graded vector spaces are
$V^{*}=\mathbb{Q}\langle v_{0},v_{2m}\rangle$, and
$W^{*}=\mathbb{Q}\langle w_{2m-1},w_{4m-1}\rangle$ with differential
$$ \partial w_{2m-1}=2v_{0}v_{2m},\quad \partial w_{4m-1}=v_{2m}^{2}. $$

The second spectral sequence is
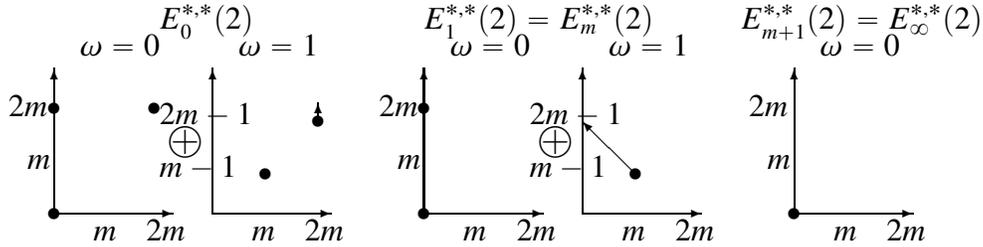
\begin{center}
\begin{picture}(360,100)
\multiput(10,20)(140,0){3}{\vector(0,1){55}}  
\put(45,57){$\bullet$}
\multiput(10,20)(140,0){3}{\vector(1,0){45}}  
\put(107,52){$\bullet$}
\multiput(70,20)(140,0){2}{\vector(0,1){55}}
\multiput(70,20)(140,0){2}{\vector(1,0){45}}
\multiput(7,17)(140,0){3}{$\bullet$} \put(50,90){$E_0^{*,*}(2)$}
\multiput(7,57)(140,0){2}{$\bullet$} 
\put(230,35){\vector(-1,1){20}}
\multiput(87,32)(140,0){2}{$\bullet$}
\put(110,55){\vector(0,1){7}}
\multiput(50,34)(140,0){2}{$m-1$}  
\multiput(50,54)(140,0){2}{$2m-1$}
\multiput(25,10)(140,0){3}{$m$} \multiput(45,10)(140,0){3}{$2m$}
\multiput(85,10)(140,0){2}{$m$}
\multiput(105,10)(140,0){2}{$2m$}
\multiput(20,80)(140,0){3}{$\omega=0$}    
\multiput(80,80)(140,0){2}{$\omega=1$}
\multiput(53,44)(140,0){2}{$\bigoplus $}      
\multiput(-7,57)(140,0){3}{$2m$} \multiput(0,37)(140,0){3}{$m$}
\put(150,90){$E_1^{*,*}(2)=E_m^{*,*}(2)$}     
\put(270,90){$E_{m+1}^{*,*}(2)=E_{\infty}^{*,*}(2)$}
\end{picture}
\end{center}
and this implies that $P_{C_2(M)}(t,s)=1$.

The third spectral sequence is
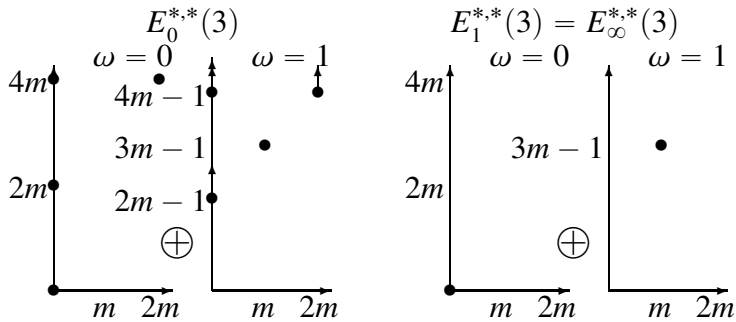
\begin{center}
\begin{picture}(360,125)
\multiput(50,20)(150,0){2}{\multiput(0,0)(60,0){2}{\vector(0,1){85}}}
\multiput(50,20)(150,0){2}{\multiput(0,0)(60,0){2}{\vector(1,0){45}}}
\multiput(33,55)(150,0){2}{$2m$}        \put(150,95){\vector(0,1){10}}
\multiput(33,95)(150,0){2}{$4m$}        \put(73,50){$2m-1$}
\multiput(73,70)(150,0){2}{$3m-1$}      \put(73,90){$4m-1$}
\multiput(65,10)(150,0){2}{\multiput(0,0)(60,0){2}{$m$}}
\multiput(83,10)(150,0){2}{\multiput(0,0)(60,0){2}{$2m$}}
\multiput(65,105)(150,0){2}{$\omega=0$} \put(197,17){$\bullet$}
\multiput(47,17)(0,40){3}{$\bullet$}    \put(277,72){$\bullet$}
\multiput(107,52)(0,40){2}{$\bullet$} 
\put(85,118){$E_0^{*,*}(3)$}
\multiput(125,105)(150,0){2}{$\omega=1$}\put(87,97){$\bullet$}
\multiput(127,72)(20,20){2}{$\bullet$}
\put(200,118){$E_1^{*,*}(3)=E_{\infty}^{*,*}(3)$}
\multiput(110,58)(0,40){2}{\vector(0,1){10}}
\multiput(90,35)(150,0){2}{$\bigoplus$}
\end{picture}
\end{center}
therefore $P_{C_3(M)}(t,s)=1+st^{4m-1}$.

By induction on $k$, we suppose that $P_{C_k}(M)(t,s)=1+st^{4m-1}$. In
the $(k+1)$-th spectral sequence we have
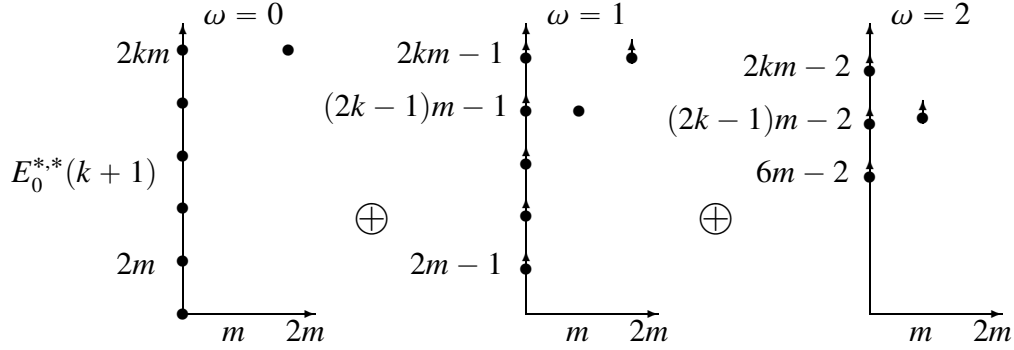
\begin{center}
\begin{picture}(360,140)
\multiput(50,20)(130,0){3}{\vector(0,1){110}}
\put(-15,70){$E_0^{*,*}(k+1)$}         \put(262,110){$2km-2$}
\multiput(50,20)(130,0){3}{\vector(1,0){50}} 
\put(58,130){$\omega=0$}
\multiput(65,10)(130,0){3}{$m$}        \put(188,130){$\omega=1$}
\multiput(89,10)(130,0){3}{$2m$}       \put(318,130){$\omega=2$}
\multiput(47,17)(0,20){6}{$\bullet$}   \put(87,117){$\bullet$}
\multiput(177,34)(0,20){5}{$\bullet$}  \put(136,35){$2m-1$}
\multiput(197,94)(20,20){2}{$\bullet$} \put(233,90){$(2k-1)m-2$}  
\multiput(307,69)(0,20){3}{$\bullet$}  \put(327,91){$\bullet$}
\multiput(180,35)(0,20){5}{\vector(0,1){9}}  \put(25,35){$2m$}
\multiput(310,70)(0,20){3}{\vector(0,1){9}}  \put(25,115){$2km$}
\multiput(330,92)(-110,23){2}{\vector(0,1){9}}
\multiput(115,53)(130,0){2}{$\bigoplus$}\put(267,70){$6m-2$}
\put(103,95){$(2k-1)m-1$}               \put(131,115){$2km-1$}
\end{picture}
\end{center}
The differential $d_0$ kills the $2m$-th column; the $0$-th
column has the cohomology of $C_k(M)$:
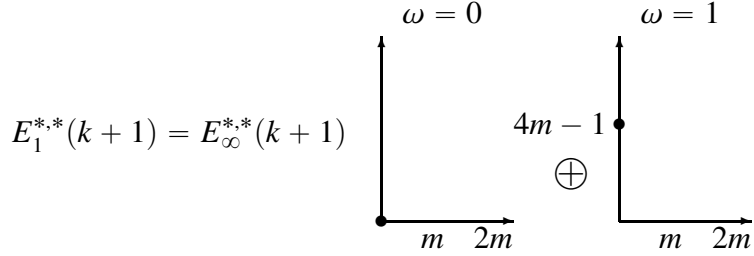
\begin{center}
\begin{picture}(360,100)
\put(0,50){$E_1^{*,*}(k+1)=E_{\infty}^{*,*}(k+1)$}
\multiput(140,20)(90,0){2}{\vector(0,1){70}} 
\put(205,35){$\bigoplus$}            \put(148,95){$\omega=0$}
\multiput(140,20)(90,0){2}{\vector(1,0){50}} 
\multiput(155,10)(90,0){2}{$m$}      \put(238,95){$\omega=1$}
\multiput(175,10)(90,0){2}{$2m$}     \put(190,53){$4m-1$}
\multiput(137,17)(90,37){2}{$\bullet$}
\end{picture}
\end{center}
hence $P_{C_{k+1}}(t,s)=1+st^{4m-1}$.
\end{proof}

The case $m=1$ in the following lemma, that is $M=\mathbb{C}\mbox{P}^{2}$, was obtained by Y. F\'{e}lix and D. Tanr\'{e} \cite{F-Ta}.
\begin{lemma}  \label{lem.2}
If $M^{4m}$ is a homology projective plane, then $M$ has the strong stability property with the range of stability 4.
\end{lemma}
\begin{proof}
The two graded spaces are $V^{*}=\mathbb{Q}\langle v_{0},v_{2m},v_{4m}\rangle$, $W^{*}=\mathbb{Q}\langle w_{4m-1},w_{6m-1},w_{8m-1}\rangle$, with differential
$$ \partial w_{4m-1}=2v_{0}v_{4m}+v_{2m}^{2},\quad \partial w_{6m-1}=2v_{2m}v_{4m},
                     \quad \partial w_{8m-1}=v_{4m}^{2}. $$
The sequence of spectral sequences starts with:
$${}^{*}E^{*,*}_{0}(1)={}^{*}E^{*,0}_{\infty}(1)\cong V^{*} $$
and
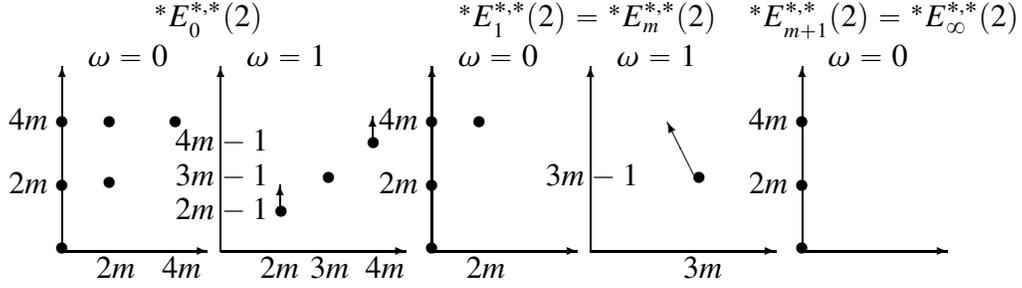
\begin{center}
\begin{picture}(360,120)
\multiput(10,20)(140,0){3}{\vector(0,1){70}}   
\put(248,45){$\bullet$}
\multiput(10,20)(140,0){3}{\vector(1,0){55}}   
\put(53,32){$2m-1$}
\multiput(70,20)(140,0){2}{\vector(0,1){70}}   \put(85,10){$2m$}
\multiput(70,20)(140,0){2}{\vector(1,0){70}}
\put(25,43){$\bullet$}   
\put(45,105){${}^{*}E_0^{*,*}(2)$}    \put(125,10){$4m$}
\multiput(7,18)(0,24){3}{$\bullet$}   
\put(249,49){\vector(-1,2){10}}
\multiput(147,18)(0,24){3}{$\bullet$} \put(104,10){$3m$}
\multiput(287,18)(0,24){3}{$\bullet$} \put(125,58){$\bullet$}
\put(108,45){$\bullet$}               \put(90,32){$\bullet$}
\multiput(-10,42)(140,0){3}{$2m$}     \put(165,66){$\bullet$}
\multiput(53,45)(140,0){2}{$3m-1$}    \put(53,58){$4m-1$}
\multiput(-10,66)(140,0){3}{$4m$}     \put(48,10){$4m$}
\multiput(23,10)(140,0){2}{$2m$}      \put(245,10){$3m$}
\multiput(20,90)(140,0){3}{$\omega=0$}\put(50,66){$\bullet$}
\multiput(80,90)(140,0){2}{$\omega=1$}\put(25,66){$\bullet$}
\put(160,105){${}^{*}E_1^{*,*}(2)={}^{*}E_{m}^{*,*}(2)$}            
\put(270,105){${}^{*}E_{m+1}^{*,*}(2)={}^{*}E_{\infty}^{*,*}(2)$}
\put(92.5,35){\vector(0,1){10}}       
\put(127.5,61){\vector(0,1){10}}
\end{picture}
\end{center}
so $\mbox{P}_{C_{1}(M)}(t,s)=\mbox{P}_{C_{2}(M)}(t,s)=1+t^{2m}+t^{4m}$. The result for the spectral sequences ${}^{*}E^{*,*}_{*}(k)$, 
$k=3,4,5$, are given in the following table
\begin{center}
Table 2
\vspace{1pt}\\
\begin{tabular}{ |c|c|}
\hline
$k$     & non-zero terms \quad ${}^{*}E^{\geq1,*}_{r}(k)={}^{*}E^{\geq1,*}_{\infty}(k)$\\
\hline
3      & ${}^{1}E^{2m,6m-1}_{1}(3)=\langle v_{4m}w_{4m-1}\rangle,$ \quad ${}^{1}E^{3m,7m-1}_{1}(3)=\langle v_{4m}w_{6m-1}\rangle$ \\
\hline
4       & ${}^{1}E^{2m,10m-1}_{1}(4)=\langle 2v_{4m}^{2}w_{4m-1}-v_{2m}v_{4m}w_{6m-1}\rangle$  \\
\hline
5       & ------- \\
\hline
\end{tabular}
\end{center}
hence 
$$ \mbox{P}_{C_{3}(M)}(t,s)=1+t^{2m}+t^{4m}+s(t^{8m-1}+t^{10m-1}) $$ 
and 
$$ \mbox{P}_{C_{4}(M)}(t,s)=\mbox{P}_{C_{5}(M)}(t,s)=1+t^{2m}+t^{4m}+s(t^{8m-1}+t^{10m-1}+t^{12m-1}). $$
From $k=6$ the spectral sequences become stable at ${}^{*}E^{*,*}_{1}$:
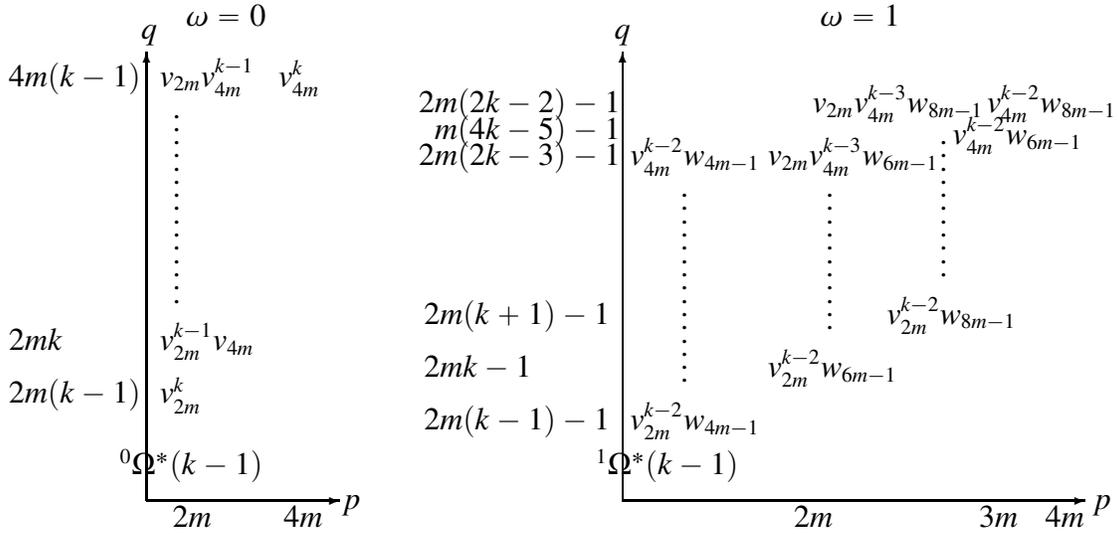
\begin{center}
\begin{picture}(400,200)
\put(60,190){$\omega=0$}             \put(300,190){$\omega=1$}
\put(35,20){${}^0\Omega^{*}(k-1)$}   \put(-7,67){$2mk$}
\put(45,10){\vector(0,1){170}}       \put(43,185){$q$}
\put(45,10){\vector(1,0){73}}        \put(119,7){$p$}
\put(225,10){\vector(0,1){170}}      \put(222,185){$q$}
\put(225,10){\vector(1,0){175}}      \put(403,7){$p$}
\put(215,20){${}^1\Omega^{*}(k-1)$}  \put(-7,47){$2m(k-1)$}                     
\put(-7,167){$4m(k-1)$}              \put(55,0){$2m$}
\put(50,47){$v_{2m}^{k}$}            \put(97,0){$4m$}
\put(50,67){$v_{2m}^{k-1}v_{4m}$}    \put(95,167){$v_{4m}^{k}$}      
\put(50,167){$v_{2m}v_{4m}^{k-1}$}   \put(360,0){$3m$}
\multiput(55,85)(0,5){15}{$\dot$}    \put(385,0){$4m$}
\put(290,0){$2m$}                    \put(150,37){$2m(k-1)-1$}   
\put(150,57){$2mk-1$}                \put(150,77){$2m(k+1)-1$}
\put(154,147){$m(4k-5)-1$}           \put(148,137){$2m(2k-3)-1$}  
\put(228,37){$v_{2m}^{k-2}w_{4m-1}$} \put(148,157){$2m(2k-2)-1$}
\put(228,137){$v_{4m}^{k-2}w_{4m-1}$}
\put(280,57){$v_{2m}^{k-2}w_{6m-1}$}  
\put(280,137){$v_{2m}v_{4m}^{k-3}w_{6m-1}$}
\multiput(247,55)(0,5){15}{$\dot$}  
\multiput(302,75)(0,5){11}{$\dot$}
\put(325,77){$v_{2m}^{k-2}w_{8m-1}$}
\multiput(345,95)(0,5){11}{$\dot$}  
\put(297,157){$v_{2m}v_{4m}^{k-3}w_{8m-1}$}
\put(350,145){$v_{4m}^{k-2}w_{6m-1}$}
\put(363,157){$v_{4m}^{k-2}w_{8m-1}$}
\end{picture}
\end{center}

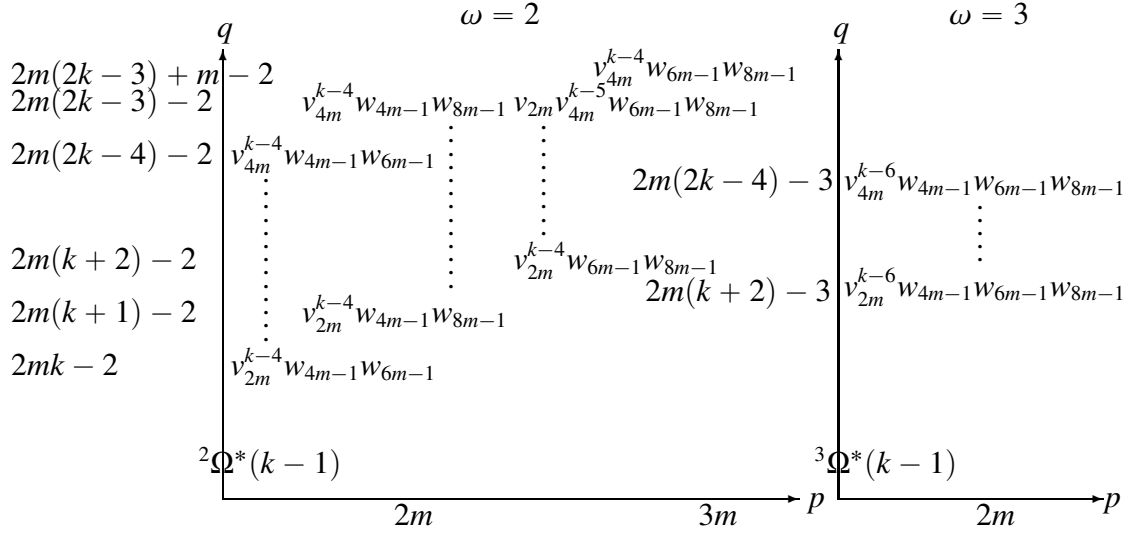
\begin{center}
\begin{picture}(400,200)
\put(160,190){$\omega=2$}            \put(345,190){$\omega=3$}
\put(70,10){\vector(0,1){170}}       \put(68,185){$q$}
\put(70,10){\vector(1,0){218}}       \put(291,7){$p$}
\put(303,10){\vector(0,1){170}}      \put(301,185){$q$}
\put(303,10){\vector(1,0){100}}      \put(403,7){$p$}
\put(61,20){${}^2\Omega^{*}(k-1)$}   \put(-10,77){$2m(k+1)-2$}                  
\put(294,20){${}^3\Omega^{*}(k-1)$}  \put(-10,57){$2mk-2$}                  
\put(73,57){$v_{2m}^{k-4}w_{4m-1}w_{6m-1}$} 
\put(73,137){$v_{4m}^{k-4}w_{4m-1}w_{6m-1}$}
\put(-10,97){$2m(k+2)-2$}            \put(-10,137){$2m(2k-4)-2$}      
\put(100,77){$v_{2m}^{k-4}w_{4m-1}w_{8m-1}$}
\put(100,157){$v_{4m}^{k-4}w_{4m-1}w_{8m-1}$}
\put(-10,157){$2m(2k-3)-2$}          \put(355,0){$2m$}     
\put(180,97){$v_{2m}^{k-4}w_{6m-1}w_{8m-1}$}
\put(-10,167){$2m(2k-3)+m-2$}        \put(250,0){$3m$}       
\put(180,157){$v_{2m}v_{4m}^{k-5}w_{6m-1}w_{8m-1}$}
\multiput(85,70)(0,5){13}{$\dot$}    \put(230,85){$2m(k+2)-3$} 
\multiput(155,90)(0,5){13}{$\dot$}   \put(225,127){$2m(2k-4)-3$}  
\multiput(190,110)(0,5){9}{$\dot$}   \put(135,0){$2m$} 
\put(210,170){$v_{4m}^{k-4}w_{6m-1}w_{8m-1}$}
\put(305,87){$v_{2m}^{k-6}w_{4m-1}w_{6m-1}w_{8m-1}$}
\put(305,127){$v_{4m}^{k-6}w_{4m-1}w_{6m-1}w_{8m-1}$}
\multiput(355,100)(0,5){5}{$\dot$}
\end{picture}
\end{center}

The differential $d_{0}$ is given by 
$d_{0}(w_{4m-1},w_{8m-1})=(v_{2m}^{2},v_{4m}^{2})$ and
$$ d_{0}(v_{2m}^{\alpha}v_{4m}^{\beta}w_{6m-1}) =\begin{cases}
      0, & \mbox{if $\alpha=0$}\\
      2v_{2m}^{\alpha+1}v_{4m}^{\beta+1} & \mbox{if $\alpha\geq 1.$}\\
   \end{cases} $$
On the column $p=0$ we get ${}^{\omega}H^*(C_{k-1})$ and nothing on the last two columns, $p=3m$ and $p=4m$; the differential $d_{0}$ is also an isomorphism in the cases:
$$ {}^{1}E_{0}^{2m,2m(k-1)-1}(k)\xlongrightarrow{\cong}{}^{0}E_{0}^{2m,2m(k-1)}(k) $$
$$\,\,{}^{1}E_{0}^{2m,2m(2k-2)-1}(k)\xlongrightarrow{\cong}{}^{0}E_{0}^{2m,2m(2k-2)}(k).$$
In general, we have the exact sequence $(j=k+2,\,k+3,\ldots,\,2k-4)$
$$ \underset{[1]}{{}^{3}E_{0}^{2m,2mj-3}(k)}\rightarrowtail \underset{[3]}{{}^{2}E_{0}^{2m,2mj-2}(k)}\xrightarrow{d_{0}} \underset{[3]}{{}^{1}E_{0}^{2m,2mj-1}(k)}\twoheadrightarrow \underset{[1]}{{}^{0}E_{0}^{2m,2mj}(k)}: $$
(the dimensions are given in the square brackets) where the matrix of $d_{0}$ is
\begin{equation*}
d_{0}=
  \begin{pmatrix}
 -2  &-1 & 0\\
  1 & 0 & -1\\
  0 & 1 & 2
  \end{pmatrix}
\end{equation*}
For small values ($j=k,k+1$) and for $j=2k-3$ we have short exact sequences
$$\,\,\underset{[1]}{{}^{2}E_{0}^{2m,2mk-2}(k)}\rightarrowtail \underset{[2]}{{}^{1}E_{0}^{2m,2mk-1}(k)}\twoheadrightarrow \underset{[1]}{{}^{0}E_{0}^{2m,2mk}(k)}$$

$$\qquad\underset{[2]}{{}^{2}E_{0}^{2m,2m(k+1)-2}(k)}\rightarrowtail \underset{[3]}{{}^{1}E_{0}^{2m,2m(k+1)-1}(k)}\twoheadrightarrow \underset{[1]}{{}^{0}E_{0}^{2m,2m(k+1)}(k)}$$

$$\,\,\qquad\underset{[2]}{{}^{2}E_{0}^{2m,2m(2k-3)-2}(k)}\rightarrowtail \underset{[3]}{{}^{1}E_{0}^{2m,2m(2k-3)-1}(k)}\twoheadrightarrow \underset{[1]}{{}^{0}E_{0}^{2m,2m(2k-3)}(k)}$$
In conclusion, we get 
$$ {}^{*}E^{*,*}_{1}(k)={}^{*}E^{*,*}_{\infty}(k)={}^{0}E^{0,*}_{\infty}(k)\oplus{}^{1}E^{0,*}_{1}(k), $$
with Poincar\'{e} polynomial $(k\geq 4)$:
$$ \mbox{P}_{C_{k}(M)}(t,s)=1+t^{2m}+t^{4m}+s(t^{8m-1}+t^{10m-1}+t^{12m-1}). $$
\end{proof}
\mbox{}\\
{\em Proof of Theorem \ref{th.2}.} Lemmas \ref{lem.1} and \ref{lem.2} give one implication of the theorem. For the opposite implication, we show in the next three lemmas that $M$ cannot have the strong stability property in the following cases:

Case 1) the Poincar\'{e} polynomial of $M^{4m}$ is $1+\beta_{2m}t^{2m}+t^{4m},\, \beta_{2m}\geq 2$;

Case 2) there is a non-zero odd Betti number $\beta_{2i+1}$;

Case 3) there is a non-zero even Betti number of $M^{2m}$, 
$\beta_{2i}$, $i\neq 0,\,\frac{m}{2},\,m$.
$\hfill \square$
\begin{lemma}  \label{lem.3}
If $M^{4m}$ has the Poincar\'{e} polynomial $1+bt^{2m}+t^{4m},$ with $b\geq 2$, then $M$ cannot have the strong stability property.
\end{lemma}
\begin{proof}
The associated graded spaces are $V^{*}=\mathbb{Q}\langle v_{0};v_{2m,1},v_{2m,2},\ldots v_{2m,b};v_{4m}\rangle$ and $W^{*}=\mathbb{Q}\langle w_{0};w_{2m,1},w_{2m,2},\ldots w_{2m,b};w_{4m}\rangle$ (although irrelevant for the argument, one can choose
the basis such that $\partial w_{4m-1}=2v_{0}v_{4m}+\sum_{i=1}^{b}v_{2m,i}^{2}$, $\partial w_{6m-1,i}=2v_{2m,i}v_{4m}$, $\partial w_{8m-1}=v_{4m}^{2}$).

In the $k$-th spectral sequence, the domain and the range of the differential
$$ d_{0}:\,{}^{1}E_{0}^{2m,2m(k-1)-1}(k)\longrightarrow {}^{0}E_{0}^{2m,2m(k-1)}(k) $$
have dimensions $\binom{k+b-3}{b-1}$ and $\binom{k+b-1}{b-1}$ respectively. Obviously $d_{0}({}^{0}E_{0}^{*,*}(k))=0$ and $d(v_{4m}^{k-2}\otimes W^{*})\subset v_{4m}^{k-2}\otimes\wedge^{2}V^{*}$, therefore we have non-zero elements in
${}^{0}E_{\infty}^{2m,2m(k-1)}(k)$ of arbitrary large degree.
\end{proof}
\begin{lemma} \label{lem.4}
If $M^{2m}$ has a non-zero odd Betti number, then $M$ cannot have the strong stability property.
\end{lemma}
\begin{proof}
Choose the non-zero odd Betti number of the highest degree, $2i+1$. The graded spaces $V^{*}$ and $W^{*}$ are
$$ V^{*}=\mathbb{Q}\langle v_{0},\ldots,v_{2i+1},v^{\prime}_{2i+1},
                                 \ldots,v_{2p},\ldots,v_{2q},\ldots,v_{2m}\rangle, $$
$$ W^{*}=\mathbb{Q}\langle w_{2m-1},\ldots,w_{2m+2i},w^{\prime}_{2m+2i},\ldots,w_{2m+2p-1},\ldots,w_{2m+2q-1},\ldots,w_{4m-1}\rangle. $$
The differential of $w_{2m+2i}$ contains a unique term, $2v_{2i+1}v_{2m},$ for degree reason (a quadratic product $v_{s}v_{t}$ with $2i+1<s,t<2m$ has even degree). The spectral sequence ${}^{k}E_{*}^{*,*}(2k+1)$ contains the product $z=v_{2i+1}w^{k}_{2m+2i}$, which is a permanent cocycle. Its is never a coboundary:
$$ d(\wedge^{*}V^{*}\otimes\wedge^{*}W^{*})\subset\wedge^{\geq2}V\otimes\wedge^{*}W^{*}. $$
The degree of $z$ is arbitrary large, therefore $M$ has not the strong stability property.
\end{proof}
\begin{remark} \label{rem.3}
C. Schliessl \cite{Sc} computed all Betti numbers of $C_{k}(\mathbb{T}^{2})$. Its top Betti number is $\beta_{\tau=k+1}(C_{k}(\mathbb{T}^{2}))=\dfrac{2k-1-3(-1)^{k}}{4}$ (see also \cite{D-K} and \cite{M}).
\end{remark}
\begin{lemma} \label{lem.5}
If $M^{2m}$ has a non-zero even Betti number $\beta_{2i},$ (with $2i\neq0,\,m$ and $2m)$, then $M$ cannot have the strong stability property.
\end{lemma}
\begin{proof}
Using Poincar\'{e} duality we can choose a positive $i$ satisfying
$0<2i<m:\quad V_{*}=\mathbb{Q}\langle v_{0},\ldots,v_{2i},\ldots,v_{2m}\rangle$. In the spectral sequence ${}^{0}E^{*,*}_{*}(2k)$, the product $v_{2i}^{2k}$ is a permanent cocycle and it is never a coboundary:
$$ d(Sym^{k-2}V^{\geq2i+1}\otimes W^{*})\subset Sym^{k-2}V^{\geq2i+1}\otimes Sym^{2}V. $$
\end{proof}
\begin{remark} \label{rem.4}
M. Maguire \cite{M} computed all Betti numbers of $C_{k}(\mathbb{C}\mbox{P}^{3})$. Its top Betti number is 
$\beta_{\tau=2k+12}(C_{k}(\mathbb{C}\mbox{P}^{3}))=1$ $(k\geq11)$.
\end{remark}


\section{Strong stability: open or non-orientable even dimensional manifolds}

In this section we use B. Knudsen's model \cite{K}: the differential graded algebra computing the cohomology of $C_{k}(M)$ for an even dimensional manifold $M^{2m}$ is given by
$$ H^{*}(\Omega^{*}(k)(V^{*},W^{*}),\partial), $$
where the graded spaces $V^{*}$ and $W^{*}$ are
$$ V^{*}=H^{-*}_{c}(M;\mathbb{Q}^{\omega})[2m],\quad W^{*}=H^{-*}_{c}(M;\mathbb{Q})[4m-1], $$
and the differential is the shifted dual of the product
$$ H^{-*}_{c}(M;\mathbb{Q}^{\omega})\otimes H^{-*}_{c}(M;\mathbb{Q}^{\omega})\longrightarrow H^{-*}_{c}(M;\mathbb{Q}) $$
(here $H^{-*}_{c}$ is cohomology with compact supports and 
$\mathbb{Q}^{\omega}$ is the orientation sheaf; as before $A^{*}[q]$ is the graded space $A^{*}$ shifted by $q$).

In the same paper B. Knudsen computed the cohomology of $C_{k}(M)$ for three even dimensional manifolds with odd cohomology: Klein bottle $\mathbb{K},$ the punctured Euclidean space $\overset{\circ}{\mathbb{R}^{n}}=\mathbb{R}^{n}\setminus\{pt\}$ and the punctured torus $\overset{\circ}{\mathbb{T}}=\mathbb{T}\setminus\{pt\}$. He found that their top Betti numbers are
$$ \beta_{\tau=k}(C_{k}(\mathbb{K}))=2, $$
$$ \beta_{\tau=k}(C_{k}(\overset{\circ}{\mathbb{R}^{n}}))=1, $$
$$ \beta_{\tau=k}(C_{k}(\overset{\circ}{\mathbb{T}}))=\frac{3+(-1)^{k+1}}{4}k+1, $$
so these three spaces do not have the strong stability property.

We will describe few cases of manifolds of even dimensions with the strong stability property.
\begin{prop}\label{tstab}
Let $M^{2m}$ be a closed non-oriented manifold with 
$\beta_{\tau}(M)\leq\left\lfloor\frac{4m-2}{3}\right\rfloor$. Then $M$ has the strong stability property if and only if $M$ is acyclic.
\end{prop}
\begin{proof}
For a closed manifold non-orientable manifold we have $H_{c}^{-*}(M)=H^{-*}(M)$ and $H^{-*}(M;\mathbb{Q}^{\omega})\cong H_{2m-*}(M)$ and these imply that
$$ V^{*}=V^{0}\oplus V^{1}\oplus\ldots \oplus V^{\beta_{\tau}(M)},\quad W^{*}=W^{4m-\beta_{\tau}(M)-1}\oplus W^{4m-\beta_{\tau}(M)}\oplus\ldots\oplus W^{4m-1}. $$
The product $V^{*}\bigotimes V^{*}\rightarrow W^{*}$ is zero by degree reason:
$$ 2\beta_{\tau}(M)<4m-\beta_{\tau}(M)-1\quad\mbox{or}\quad \beta_{\tau}(M)\leq\frac{4m-2}{3}, $$
hence the differential $\partial$ is also zero.
If $M^{2m}$ is not acyclic, then there is a non-zero $x\in H^{\geq1}(M).$ If the degree of $x$ is even, then there is a corresponding non-zero $v\in V^{even\geq 2}$; otherwise, there is a corresponding non-zero $w\in W^{even}$. Therefore $v^{k},$ respectively $w^{k}$, are non-zero cohomology classes of arbitrary large degree, and this contradicts the strong stability property.
\end{proof}
\begin{prop} \label{prop.2}
Let $M^{2m}$ be a closed non-orientable manifold with odd cohomology. Then $M$ has the strong stability property if and only if $M$ is acyclic.
\end{prop}
\begin{proof}
For a closed non-orientable manifold we have $H^{-*}_{c}(M)=H^{-*}(M),$ $H^{-2m}(M)=0$, $H^{0}(M)=\mathbb{Q}$, hence $W^{*}=W^{\geq2m}$ and
$W^{4m-1}=\mathbb{Q}\langle w_{4m-1}\rangle$.

If $M$ has odd cohomology, we have
$$ W^{*}=W^{even}\oplus W^{4m-1} $$
by Poincar\'{e} duality (see \cite{B} or \cite{D}),
$H^{-*}(M^{2m};\mathbb{Q}^{\omega})\cong H_{2m-*}(M;\mathbb{Q})$, hence
$$ V^{*}=V^{\leq 2m-1}=\mathbb{Q}\langle v_{0}\rangle\oplus V^{odd}.$$
A non-zero (odd) Betti number of $M^{2m}$ will give a non-zero $w\in W^{even}$. Its differential $\partial w$ is in $(\bigwedge^{2}V)^{odd}=v_{0}\wedge V^{odd}$; the degree of $w$ is at least $2m$, and the degree of an element in $v_{0}\wedge V^{odd}$ is at most $2m-1$, therefore $\partial w=0$. The product $w^{k}$ gives a permanent cocycle in $E^{*,*}_{*}(2k)$, and it is never a coboundary:
$$ \partial(Sym V^{*}\otimes Sym W^{*})\subset Sym^{\geq2}V^{*}\otimes Sym W^{*}. $$
The degree of $w^{k}$ is arbitrary large, hence $M$ cannot have the strong stability property.

If $M^{2m}$ is acyclic, the cohomology of $C_{k}(M)$ is reduced to
$$ H^{*}(C_{k}(M))\cong\mathbb{Q}\langle v_{0}^{k},v_{0}^{k-2}w_{4m-1}\rangle $$
and this is stable.
\end{proof}
\begin{prop} \label{prop.3}
Let $M^{2m}$ be an open orientable manifold with odd cohomology. Then $M$ has the strong stability property if and only if $M$ is acyclic.
\end{prop}
\begin{proof}
For an open oriented manifold $M^{2m}$ we have, by Poincar\'{e} duality,
$$ H^{-*}_{c}(M;\mathbb{Q}^{\omega})=H^{-*}_{c}(M;\mathbb{Q})\cong H_{2m-*}(M;\mathbb{Q}),\,H_{2m}(M)=0,\,H_{0}(M)=\mathbb{Q}, $$
hence $W^{*}=\mathbb{Q}\langle w_{2m-1}\rangle\oplus W^{\geq2m}$. If $M$ has odd cohomology, we also have
$$ W^{*}=\mathbb{Q}\langle w_{2m-1}\rangle\oplus W^{even} \mbox{ and } V^{*}=
          \mathbb{Q} \langle v_{0}\rangle\oplus V^{odd\leq2m-1}. $$
Now we can repeat the argument in the proof of Proposition \ref{prop.2}.
\end{proof}
\begin{remark} \label{rem.14}
It seems that the sequence of Betti numbers 
$\{\beta_i(C_{k}(M)\}_{k\geq 1}$ is increasing for any 
$i\geq 0$ and for any manifold $M$, with the exception of 
$S^{2m}$. For other peculiar properties of the cohomology of configuration spaces of $S^2$, see \cite{AAB} and \cite{BMPP}.
\end{remark}


\section{Shifted stability}
We start to analyse the odd dimensional case.
\begin{prop}
A manifold $M^{2m+1}$ satisfies the shifted stability condition if and only if the top positive even Betti number is one.
\end{prop}
\begin{proof}
Let $\beta_{2a}$ the top even Betti number $(a\geq1).$ For $k\geq\Sigma_{i>2a}\,\beta_{i}+1=\beta+1$ we have
$$ H^{top}(C_{k}(M))=Sym^{k-\beta}V^{2a}\bigotimes\bigwedge^{\beta}V^{\geq 2a+1}, $$
hence $M$ has the shifted stability property if and only if $\mbox{dim}\,Sym^{k-\beta}V^{2a}$ does not depend on $k$, therefore 
$\beta_{2a}$ should be one.
\end{proof}
Now, for an even dimensional manifold $M$, we give three new definitions for shifted stability of the sequence 
$\{C_{k}(M)\}_{k\geq1}$. In the first definition, the spectral sequences $\{{}^{*}E_{*}^{*,*}(k)\}_{k\geq1}$ are those defined in Section 2. We suppose that $M$ satisfies the condition
${}^{\omega}E_{\infty}^{0,*}(k+1)={}^{\omega}H^{*}(C_{k}(M))$.
\begin{definition}
The manifold $M$ satisfies the {\em spectral shifted stability condition} with {\em range } $r$ and {\em shift} $\sigma$ $(r,\sigma\geq1)$ if and only if, for any $k\geq r$, any $p\geq1$ and any 
$\omega\geq0$, we have
$$ {}^{\omega}E_{\infty}^{p,q+\sigma}(k+1)={}^{\omega}E_{\infty}^{p,q}(k)\mbox{ and this is non-zero}. $$
\end{definition}
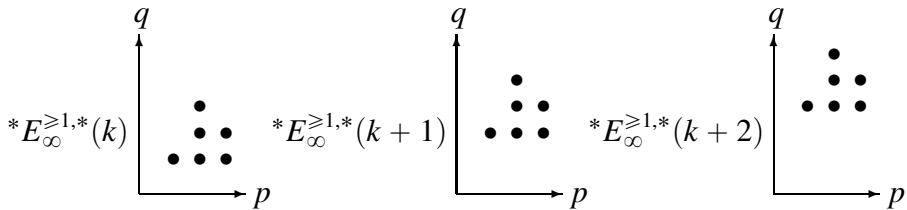
\begin{center}
\begin{picture}(360,80)
\multiput(40,10)(120,0){3}{\vector(0,1){60}}
\put(-10,30){${}^{*}E_\infty^{\geq1,*}(k)$}   
\multiput(38,75)(120,0){3}{$q$} \multiput(83,7)(120,0){3}{$p$}
\put(90,30){${}^{*}E_\infty^{\geq1,*}(k+1)$}  
\put(210,30){${}^{*}E_\infty^{\geq1,*}(k+2)$}
\multiput(40,10)(120,0){3}{\vector(1,0){40}}
\multiput(50,20)(120,10){3}{\multiput(0,0)(10,0){3}{$\bullet$} \put(10,20){$\bullet$}   \multiput(10,10)(10,0){2}{$\bullet$}}
\end{picture}
\end{center}
\begin{definition}
The manifold $M$ satisfies the {\em Poincar\'{e} polynomial shifted stability condition} with {\em range} $r$, {\em shift} $\sigma$ $(r,\sigma\geq1)$ and {\em ratio} $R(s,t)\neq0$  if and only if, for any $k\geq r$, we have
$$ P_{C_{k+1}(M)}(t,s)=P_{C_{k}(M)}(t,s)+t^{(k+1-r)\sigma}R(t,s). $$
\end{definition}
\begin{definition}
The manifold $M$ satisfies the {\em extended shifted stability condition} with {\em rang} $r$ and {\em shift} $\sigma$ $(r,\sigma\geq1)$ if and only if, for any $k\geq r$, we have
$$ P_{C_{k+1}(M)}^{[(k-r+1)\sigma]}(t,s)=t^{\sigma}P_{C_{k}(M)}^{[(k-r+1)\sigma]}(t,s). $$
\end{definition}
The relation between these shifted stability properties are given by:
\begin{prop}
Spectral shifted stability $\Rightarrow$ Poincar\'{e} polynomial shifted stability $\Rightarrow$  extended shifted stability 
$\Rightarrow$ shifted stability.
\end{prop}
\begin{proof} First implication: Let us define the polynomial $R(s,t)$, the ratio of an arithmetical sequence, as the 
two-variables Poincar\'{e} polynomial of ${}^{\omega}E_{\infty}^{\geq 1,*}(r)$:
$$ R(s,t)=\sum_{\omega\geq0}\sum_{\substack{p+q=i\\p\geq1}}\mbox{dim}{}^{\omega}E_{\infty}^{p,q}(r)t^{i}s^{\omega}. $$
By induction we get
$$ {}^{\omega}E_{\infty}^{p,q}(r)={}^{\omega}E_{\infty}^{p,q+(k-r)\sigma}(k)\qquad (\mbox{for a positive } p \mbox{ and } k\geq r) $$
and this implies that Poincar\'{e} polynomial of ${}^{\omega}E_{\infty}^{\geq 1,*}(k)$ is constant, for $k\geq r$, up to a shift with $t^{\sigma}$. Therefore we have:
\begin{align*}
P_{C_{k+1}(M)}(t,s)&=P_{{}^{*}E_{\infty}^{0,*}(k+1)}(t,s)+P_{{}^{*}E_{\infty}^{\geq1,*}(k+1)}(t,s)\\
&=P_{C_{k}(M)}(t,s)+t^{(k+1-r)\sigma}R(t,s),
\end{align*}
hence the spectral shifted stability condition with range $r$ and shift $\sigma$ gives the Poincar\'{e} polynomial shifted stability condition with the same range $r$ and shift $\sigma$.

Second implication: The recurrence formula $P_{C_{k+1}}=P_{C_{k}}+t^{(k+1-r)\sigma}R$ gives
$$ P_{C_{k}(M)}(t,s)=P_{C_{r}(M)}(t,s)+(t^{\sigma}+t^{2\sigma}+\ldots+t^{(k-r)\sigma})R(t,s). $$
Take $\rho$ such that the strip $[0,\rho]\times\mathbb{R}$ contains the support of $P_{C_{r}(M)}(t,s)$ and $h$ big enough such that support of $t^{h\sigma}R(s,t)$ is contained in 
$[\rho+1,\infty)\times\mathbb{R}$.
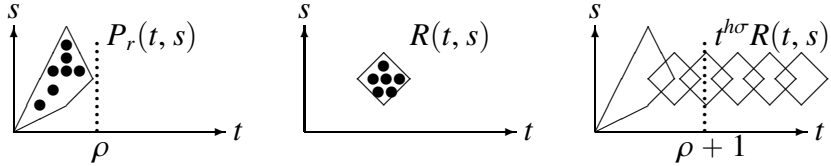
\begin{center}
\begin{picture}(360,60)
\multiput(30,10)(110,0){3}{\vector(0,1){40}}    
\multiput(27,53)(110,0){3}{$s$}      
\multiput(30,10)(110,0){3}{\vector(1,0){80}}    
\multiput(113,7)(110,0){3}{$t$}      
\multiput(30,10)(220,0){2}{\put(0,0){\line(1,2){20}}  
\put(20,10){\line(1,1){10}}      \put(0,0){\line(2,1){20}}            \put(20,40){\line(1,-2){10}}}        
\put(160,30){\line(1,-1){10}}    \put(160,30){\line(1,1){10}}
\put(180,30){\line(-1,-1){10}}   \put(180,30){\line(-1,1){10}}
\multiput(270,30)(12,0){5}{\put(0,0){\line(1,-1){10}} 
\put(0,0){\line(1,1){10}}
\put(20,0){\line(-1,1){10}}      \put(20,0){\line(-1,-1){10}}}
\multiput(42,30)(5,0){3}{$\bullet$}  \put(42,23){$\bullet$}
\multiput(47,35)(0,5){2}{$\bullet$}  \put(167,32){$\bullet$}
\multiput(163,27)(5,0){3}{$\bullet$} \put(65,43){$P_{r}(t,s)$} 
\multiput(165,22)(5,0){2}{$\bullet$} \put(60,1){$\rho$}
\put(281,1){$\rho+1$}                \put(37,17){$\bullet$}
\multiput(60,10)(0,3){12}{$\dot$}    \put(180,43){$R(t,s)$}
\multiput(290,10)(0,3){12}{$\dot$}   \put(295,43){$t^{h\sigma}R(t,s)$}
\end{picture}
\end{center}
For $\rho=r+h-1$ we have $P^{[\sigma]}_{\rho+1}(t,s)=t^{\sigma}P_{\rho}^{[\sigma]}(t,s)$, next we have $P^{[2\sigma]}_{\rho+2}(t,s)=t^{\sigma}P_{\rho+1}^{[2\sigma]}(t,s)$, and in general, for 
$k\geq\rho,$
$$ P^{[(k+1-\rho)\sigma]}_{k+1}(t,s)=t^{\sigma}P_{k}^{[(k+1-\rho)\sigma]}(t,s). $$

Third implication: This is obvious.
\end{proof}
\begin{remark}
In order to have ``weight stability at 0'' in the sequence of spectral sequences $\{{}^{*}E_{*}^{*,*}(k)\}_{k\geq1}$ (i.e, there is a range $r$ and a weight $\omega_{max}$ such that ${}^{\omega}E_{0}^{*,*}(k)=0$ for any $k\geq r$ and any $\omega>\omega_{max}$), we have to consider only manifolds with even cohomology: a non-zero odd cohomology class $x\in H^{odd}(M)$ will give a non-zero $\omega\in W^{even}$ and infinitely many non-zero terms $\omega^{s}\in{}^{s}E_{0}^{*,*}(2s)$ of arbitrary large weights.
\end{remark}
In fact, if the manifold $M$ has the spectral shifted stability property, then $M$ should have even cohomology.
\begin{prop}\label{even}
If there is a non-zero cohomology class $x\in H^{odd}(M)$, then $M$ cannot have the spectral shifted stability property.
\end{prop}
\begin{proof}
Take a maximal odd degree element $v_{2i+1}\in V^{*}=\langle v_{0},\ldots,v_{2m}\rangle$ and the corresponding $w_{2m+2i}\in W^{*}$. The relations
$$ d(w_{2m+2i})=2v_{2i+1}v_{2m}\quad\text{ and }\quad d(w_{4m-1})=v_{2m}^{2} $$
give the infinite (non-zero) cocycle $$2hv_{2m+1}w_{2m+2i}^{h-1}w_{4m-1}+v_{2m}w_{2m+2i}^{h}\in{}^{h}E_{\infty}^{*,*}(2h+1)$$ of arbitrary large weight. Definitely, the spectral shifted stability condition implies the ``weight stability condition at $\infty$'': there is a range $r$ and a weight $\omega_{max}$ such that 
${}^{\omega}E_{\infty}^{*,*}(k)=0$ for $k\geq r$ and $\omega>\omega_{max}$.
\end{proof}
For the Poincar\'{e} polynomial shifted stability condition, a weaker condition is needed.
\begin{prop}\label{Euler}
If $\chi(M)\leq-2,$ then the manifold $M$ does not satisfy the Poincar\'{e} polynomial shifted stability condition.
\end{prop}
\begin{proof}
The recurrence relations $(k\geq r)$
\begin{align*}
P_{C_{k+1}(M)}(t,s)&=P_{C_{k}(M)}(t,s)+t^{(k+1-r)\sigma}R(t,s),\\
P_{C_{k+2}(M)}(t,s)&=P_{C_{k+1}(M)}(t,s)+t^{(k+2-r)\sigma}R(t,s)
\end{align*}
imply that, for large $k$, $\chi(C_{k}(M))$ is an arithmetic sequence (if $\sigma$ is even) or $\chi(C_{k}(M))=\chi(C_{k+2}(M))=\chi(C_{k+4}(M))=\ldots$ (if $\sigma$ is odd).

If $\chi(M)\leq-2,$ the Euler characteristics $\{\chi(C_{k}(M))\}_{k\geq1}$, that is the coefficients in the expansion of $(1+t)^{\chi(M)}$, have a polynomial growth (at least quadratic) for $\chi(M)\leq-3$ and, for $\chi(M)=-2$, $\chi(C_{k}(M))=(-1)^{k+1}(k+1)$.
\end{proof}
In the case of a manifold $M^{2m}$ with Poincar\'{e} polynomial shifted stability, Propositions \ref{stab} and \ref{hstab} give some restriction for the shift $\sigma$ and ratio $R(t,s)$. For instance, we have:
\begin{prop}\label{ineq}
For a $(h-1)$-connected closed orientable manifold $M^{2m}$ satisfying the Poincar\'{e} polynomial shifted condition with shift $\sigma,$ we have the inequality $h\leq\sigma$.
\end{prop}
\begin{proof}
Choose $j\geq0$ such that there is a non-zero coefficient $r^{i,j}$ of the ratio polynomial $R(s,t).$ From Proposition \ref{hstab}
$$ (k+1-r)\sigma+i \geq\begin{cases}
      h(k+1) & \mbox{if }i=0\\
      hk+(2m-h-1)j & \mbox{if }i\geq1,\\
   \end{cases}
$$
and, for large $k$, this implies $\sigma\geq h$.
\end{proof}
\begin{remark}
In the following examples, $\mathbb{C}\mbox{P}^{1}\times\mathbb{C}\mbox{P}^{1}$ and $\mathbb{C}\mbox{P}^{3}$, we have $h=\sigma=2$; the same is true for $\mathbb{C}\mbox{P}^{4}$. For $\mathbb{C}\mbox{P}^{5}$ and $\mathbb{C}\mbox{P}^{6}$ we have $h=2,$ $\sigma=4$.
\end{remark}

The shifted stability property gives a formula for cd$(k)$, the cohomological dimension of $C_{k}(M)$.
\begin{prop}\label{cdim}
For a manifold $M$ satisfying the shifted stability condition with range $r$ and shift $\sigma$ we have, for any $k\geq r$,
$$ {\rm cd}(k)={\rm cd}(r)+(k-r)\sigma.$$
\end{prop}
\begin{proof}
This is clear from the definition.
\end{proof}
\begin{example}
The cohomological dimension of $C_{k}(\mathbb{C}\mbox{P}^{1}\times\mathbb{C}\mbox{P}^{1})$ is given by
$$ \mbox{cd}(1)=\mbox{cd}(2)=4,\,\,\mbox{cd}(3)=9,\,\,\mbox{cd}(4)=11,\,\,\mbox{cd}(k)=2k+4\,\mbox{if }k\geq5. $$
\end{example}
For large $k$, the classical stability property and the extended shifted stability property give all the Betti numbers of $C_{k}(M)$.
\begin{prop}\label{total}
Let $M^{2m}$ be a $(h-1)$-connected closed orientable manifold satisfying the extended shifted stability condition with the range $r$ and shift $\sigma.$ Then, for any $k$ satisfying the inequality $\mbox{max}\{r,{\rm cd}(r)\}\leq hk+h-2$, we have the recurrence relation
$$ H^{*}(C_{k+1}(M))=H^{\leq {\rm cd}(r)}(C_{k}(M))\bigoplus H^{\geq {\rm cd}(r)-\sigma+1}(C_{K}(M))[\sigma]. $$
\end{prop}
\begin{proof}
If ${\rm cd}(r)\leq hk+h-2$, from Corollary \ref{htot}, we have the initial equality
$$ H^{\leq {\rm cd}(r)}(C_{k}(M))=H^{\leq {\rm cd}(r)}(C_{k}(M)) $$
and, if $r\leq k$, using the extended shifted stability property, we have the final equality
$$ H^{\geq {\rm cd}(r)+1}(C_{k+1}(M))= H^{\geq {\rm cd}(r)-\sigma+1}(C_{k}(M))[\sigma]. $$
\end{proof}


\section{Shifted stability: examples}

The first example is the product $\mathbb{C}\mbox{P}^{1}\times\mathbb{C}\mbox{P}^{1}$. In this case the graded spaces $V^{*},$ $W^{*}$ and the differential $\partial_{\omega}$ are
$$ V^{*}=\langle v_{0},v_{2},\bar{v}_{2},v_{4}\rangle,\quad W^{*}=\langle w_{3},w_{5},\bar{w}_{5},w_{7}\rangle $$
$$ \partial_{\omega}( w_{3},w_{5},\bar{w}_{5},w_{7})=(2v_{0}v_{4}+v_{2}^{2},2v_{2}v_{4},2\bar{v}_{2}v_{4},v_{4}^{2}). $$
New cocycles in ${}^{*}E_{0}^{*,*}(k),$ $k\geq3,$ are generated by
\begin{align*}
\gamma&=v_{2}\bar{w}_{5}-v_{4}w_{3},\quad \bar{\gamma}=\bar{v}_{2}w_{5}-v_{4}w_{3},\\
\varepsilon&=v_{4}w_{5}-v_{2}w_{7},\quad \bar{\varepsilon}=v_{4}\bar{w}_{5}-\bar{v}_{2}w_{7},\\
\eta&=v_{4}w_{5}\bar{w}_{5}-\bar{v}_{2}w_{5}w_{7}+v_{2}\bar{w}_{5}w_{7}\quad(\bar{\eta}=-\eta).
\end{align*}
\begin{prop}
The product of two projective lines, $\mathbb{C}\emph{P}^{1}\times\mathbb{C}\emph{P}^{1}$, satisfies the spectral shifted stability condition with range $6$ and shift $2$. More precisely, the non-zero pieces ${}^{*}E_{\infty}^{\geq1,*}(k)$ (for
$k\geq6$) are:
\begin{align*}
{}^{0}E_{\infty}^{2,2k-2}(k)&=\langle v_{2}^{k},\bar{v}_{2}^{k}\rangle,\\
{}^{1}E_{\infty}^{2,2k-1}(k)&=\langle v_{2}^{k-3}\gamma,\bar{v}_{2}^{k-3}\bar{\gamma}\rangle,
{}^{1}E_{\infty}^{2,2k+1}(k)=\langle v_{2}^{k-3}\varepsilon,\bar{v}_{2}^{k-3}\bar{\varepsilon}\rangle,\\
{}^{2}E_{\infty}^{2,2k+2}(k)&=\langle v_{2}^{k-5}\eta,\bar{v}_{2}^{k-5}\bar{\eta}\rangle.
\end{align*}
\end{prop}
\begin{proof}
The sequence of spectral sequences starts with
$$ {}^{*}E_{0}^{*,*}(1)={}^{*}E_{\infty}^{*,0}(1)\cong V^{*} $$
and
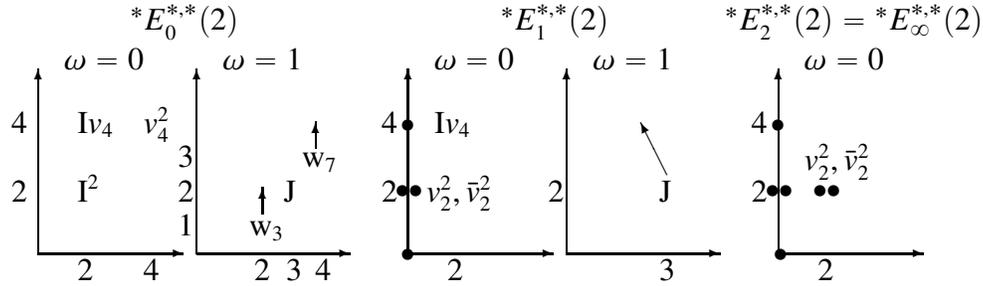
\begin{center}
\begin{picture}(360,120)
\multiput(10,20)(140,0){3}{\vector(0,1){70}} \put(245,40){$\mbox{J}$}
\multiput(10,20)(140,0){3}{\vector(1,0){55}} \put(63,27){$1$}
\multiput(70,20)(140,0){2}{\vector(0,1){70}} \put(92,10){$2$}
\multiput(70,20)(140,0){2}{\vector(1,0){58}} \put(63,53){$3$}  
\put(45,105){${}^{*}E_0^{*,*}(2)$}
\multiput(147,66)(140,0){2}{$\bullet$}  \put(147,17){$\bullet$}      
\put(248,50){\vector(-1,2){10}}
\multiput(145,41)(140,0){2}{$\bullet$}       \put(104,10){$3$}
\multiput(150,41)(140,0){2}{$\bullet$}  
\put(110,53){$\mbox{w}_{7}$}
\multiput(303,41)(5,0){2}{$\bullet$}    \put(288,17){$\bullet$}  
\put(300,52){$v_2^{2},\bar{v}_2^{2}$}
\put(157,40){$v_2^{2},\bar{v}_2^{2}$}   \put(103,40){$\mbox{J}$}
\multiput(0,40)(140,0){3}{$2$}    \put(160,66){$\mbox{I}v_{4}$}
\multiput(63,40)(140,0){2}{$2$}   \multiput(0,66)(140,0){3}{$4$}                
\multiput(25,10)(140,0){3}{$2$}   \put(245,10){$3$}
\multiput(50,10)(65,0){2}{$4$}    \put(90,27){$\mbox{w}_{3}$}
\multiput(20,90)(140,0){3}{$\omega=0$}  \put(50,66){$v_{4}^{2}$}
\multiput(80,90)(140,0){2}{$\omega=1$}
\put(25,66){$\mbox{I}v_{4}$}
\put(185,105){${}^{*}E_1^{*,*}(2)$} \put(25,40){$\mbox{I}^{2}$}
\put(270,105){${}^{*}E_{2}^{*,*}(2)={}^{*}E_{\infty}^{*,*}(2)$}
\put(95,35){\vector(0,1){10}}     \put(115,60){\vector(0,1){10}}
\end{picture}
\end{center}

Here and in the following computations we use the notation
$$ \mbox{I}=\langle v_{2},\bar{v}_{2}\rangle,\,\, \mbox{J}=\langle w_{5},\bar{w}_{5}\rangle \mbox{ and also }\mbox{I}^{k}=\langle v_{2}^{k},v_{2}^{k-1}\bar{v}_{2},\ldots,\bar{v}_{2}^{k}\rangle, $$
$$ \mbox{J}^{2}=\langle w_{5},\bar{w}_{5},\rangle,\,\,\mbox{IJ}=\langle v_{2}w_{5},\bar{v}_{2}w_{5},v_{2}\bar{w}_{5},\bar{v}_{2}\bar{w}_{5}\rangle\mbox{ and so on.} $$
The results for the spectral sequences ${}^{*}E_{*}^{*,*}(k),$ $k=3,4,\ldots,7,$ the ``weight unstable part,'' are given in the table
($\triangle_{k}=P_{C_{k}}(t,s)-P_{C_{k-1}}(t,s)$):
\newpage
\begin{center}
Table 3
\vspace{1pt}\\
\begin{tabular}{|c|c|c|}
\hline
$k$       & non-zero terms ${}^{*}E_{r}^{\geq1,*}(k)={}^{*}E_{\infty}^{\geq1,*}(k)$ & $\triangle_{k}$ \TBstrut\\
\hline
3      &$ \begin{array}{l} \qquad\quad{}^{0}E_{1}^{2,4}=\langle v_{2}^{3},\bar{v}_{2}^{3}\rangle, \Tstrut\\[0.3em]
{}^{1}E_{1}^{2,5}=\langle \gamma,\bar{\gamma}\rangle, {}^{1}E_{1}^{2,7}=\langle \varepsilon,\bar{\varepsilon}\rangle \end{array} $ &  $2t^{6}+s(2t^{7}+2t^{9})$ \Bstrut \\
\hline
4       &$ \begin{array}{l}
\qquad\qquad\qquad{}^{0}E_{1}^{2,6}=\langle v_{2}^{4},\bar{v}_{2}^{4}\rangle, \Tstrut\\[0.3em]
 {}^{1}E_{1}^{2,7}=\langle v_{2}\gamma,\bar{v}_{2}\bar{\gamma}\rangle, {}^{1}E_{2}^{2,9}=\langle v_{2}\varepsilon,\bar{v}_{2}\bar{\varepsilon},v_{2}\bar{\varepsilon}\,(=-\bar{v}_{2}\varepsilon)\rangle \end{array} $&  $2t^{8}+s(2t^{9}+3t^{11})$ \Bstrut  \\
\hline
5       & $ \begin{array}{l} \qquad\qquad{}^{0}E_{1}^{2,8}=\langle v_{2}^{5},\bar{v}_{2}^{5}\rangle,\Tstrut\\[0.3em]
{}^{1}E_{1}^{2,9}=\langle v_{2}^{2}\gamma,\bar{v}_{2}^{2}\bar{\gamma}\rangle, {}^{1}E_{1}^{2,11}=\langle v_{2}^{2}\varepsilon,\bar{v}_{2}\bar{\varepsilon}\rangle,\Tstrut\\[0.3em]
\qquad\qquad{}^{2}E_{1}^{3,11}=\langle v_{4}\mbox{J}^{2}\rangle\end{array} $& $2t^{10}+s(2t^{11}+2t^{13})+s^{2}t^{14}$ \Bstrut  \\
\hline
6       & $ \begin{array}{l} \qquad\qquad{}^{0}E_{1}^{2,10}=\langle v_{2}^{6},\bar{v}_{2}^{6}\rangle,\Tstrut\\[0.3em]
{}^{1}E_{1}^{2,11}=\langle v_{2}^{3}\gamma,\bar{v}_{2}^{3}\bar{\gamma}\rangle, {}^{1}E_{1}^{2,11}=\langle v_{2}^{3}\varepsilon,\bar{v}_{2}^{3}\bar{\varepsilon}\rangle,\Tstrut\\[0.3em]
 \qquad\qquad{}^{2}E_{1}^{2,14}=\langle v_{2}\eta,\bar{v}_{2}\bar{\eta}\rangle\end{array} $& $2t^{12}+s(2t^{13}+2t^{15})+2s^{2}t^{16}$ \Bstrut  \\
\hline
7       & $ \begin{array}{l}\qquad\qquad {}^{0}E_{1}^{2,12}=\langle v_{2}^{7},\bar{v}_{2}^{7}\rangle,\Tstrut\\[0.3em]
{}^{1}E_{1}^{2,13}=\langle v_{2}^{4}\gamma,\bar{v}_{2}^{4}\bar{\gamma}\rangle, {}^{1}E_{1}^{2,15}=\langle v_{2}^{4}\varepsilon,\bar{v}_{2}^{4}\bar{\varepsilon}\rangle,\Tstrut\\[0.3em]
\qquad\qquad{}^{2}E_{1}^{2,16}=\langle v_{2}^{2}\eta,\bar{v}_{2}^{2}\bar{\eta}\rangle\end{array} $& $2t^{14}+s(2t^{15}+2t^{17})+2s^{2}t^{18}$ \Bstrut \\
\hline
\end{tabular}
\end{center}
For $k\geq8$ the sequence $\{{}^{*}E_{*}^{*,*}(k)\}$ is ``weight stable at 0'' $({}^{\geq5}E_{0}^{*,*}(k)=0)$ and we have the following picture of the first page of the $k$-th term ${}^{*}E_0^{*,*}(k)$:

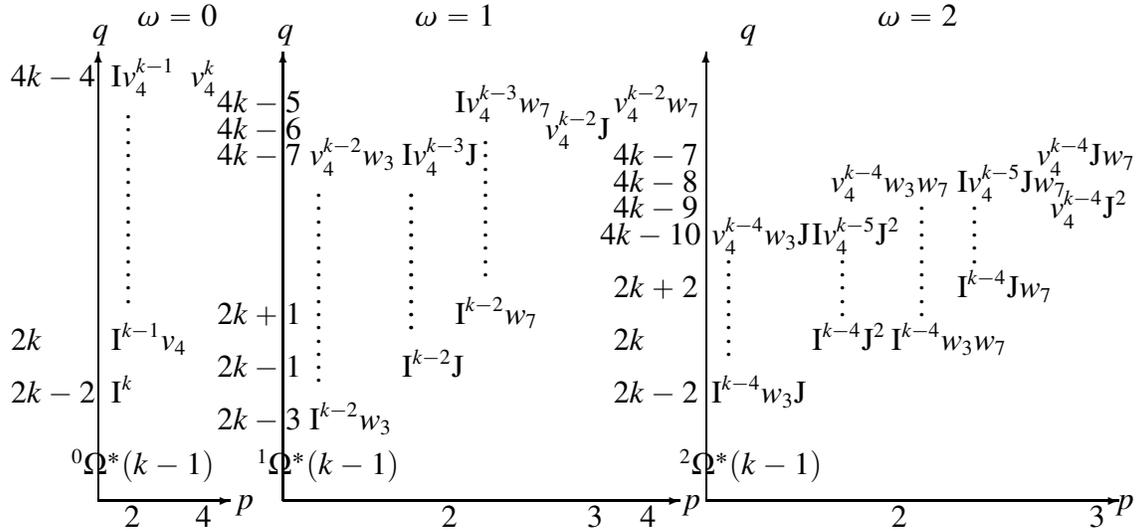
\begin{center}
\begin{picture}(398,200)
\put(35,190){$\omega=0$}            \put(140,190){$\omega=1$}
\put(10,20){${}^0\Omega^{*}(k-1)$}  \put(315,190){$\omega=2$}
\put(20,10){\vector(0,1){170}}      \put(18,185){$q$}
\put(20,10){\vector(1,0){50}}       \put(72,7){$p$}
\put(90,10){\vector(0,1){170}}      \put(88,185){$q$}
\put(90,10){\vector(1,0){150}}      \put(241,7){$p$}
\put(250,10){\vector(0,1){170}}     \put(263,185){$q$}
\put(250,10){\vector(1,0){156}}     \put(405,7){$p$}
\put(80,20){${}^1\Omega^{*}(k-1)$}  \put(-13,67){$2k$}
\put(240,20){${}^2\Omega^{*}(k-1)$} \put(-13,47){$2k-2$}                     
\put(-13,167){$4k-4$}               \put(30,0){$2$}
\put(25,47){$\mbox{I}^{k}$}         \put(57,0){$4$}
\put(150,0){$2$}            \multiput(205,0)(190,0){2}{$3$}
\put(225,0){$4$}            \put(25,67){$\mbox{I}^{k-1}v_{4}$}        \put(25,167){$\mbox{I}v_{4}^{k-1}$} \put(55,167){$v_{4}^{k}$} 
\put(155,77){$\mbox{I}^{k-2}w_{7}$} \put(65,37){$2k-3$}   
\multiput(30,85)(0,5){15}{$\dot$}   \put(65,57){$2k-1$} 
\put(100,37){$\mbox{I}^{k-2}w_{3}$} 
\put(100,137){$v_{4}^{k-2}w_{3}$}
\put(135,57){$\mbox{I}^{k-2}\mbox{J}$} \put(65,137){$4k-7$}                
\put(135,137){$\mbox{I}v_{4}^{k-3}\mbox{J}$} 
\put(65,157){$4k-5$}
\multiput(102,55)(0,5){15}{$\dot$}        \put(65,147){$4k-6$}  
\multiput(137,75)(0,5){11}{$\dot$}        \put(65,77){$2k+1$}
\put(155,157){$\mbox{I}v_{4}^{k-3}w_{7}$} \put(320,0){$2$}
\put(189,147){$v_{4}^{k-2}\mbox{J}$}      \put(215,47){$2k-2$}   
\multiput(165,95)(0,5){11}{$\dot$}  
\put(215,157){$v_{4}^{k-2}w_{7}$}
\put(252,47){$\mbox{I}^{k-4}w_{3}\mbox{J}$}  \put(215,67){$2k$}                       
\put(290,67){$\mbox{I}^{k-4}\mbox{J}^2$}  \put(215,87){$2k+2$}                     
\put(320,67){$\mbox{I}^{k-4}w_{3}w_{7}$}  \put(210,107){$4k-10$}                   
\put(345,87){$\mbox{I}^{k-4}\mbox{J}w_{7}$}\put(215,117){$4k-9$}                    
\put(252,107){$v_4^{k-4}w_3\mbox{J}$}      \put(215,127){$4k-8$}   
\multiput(257,65)(0,5){8}{$\dot$}
\multiput(300,80)(0,5){5}{$\dot$}
\multiput(330,80)(0,5){9}{$\dot$}   
\multiput(350,100)(0,5){5}{$\dot$}
\put(290,107){$\mbox{I}v_4^{k-5}\mbox{J}^2$} 
\put(215,137){$4k-7$}                    
\put(297,127){$v_4^{k-4}w_3w_7$}    
\multiput(137,75)(0,5){11}{$\dot$}       
\put(345,127){$\mbox{I}v_4^{k-5}\mbox{J}w_7$}
\put(380,117){$v_4^{k-4}\mbox{J}^2$} 
\put(375,137){$v_4^{k-4}\mbox{J}w_7$}
\end{picture}
\end{center}

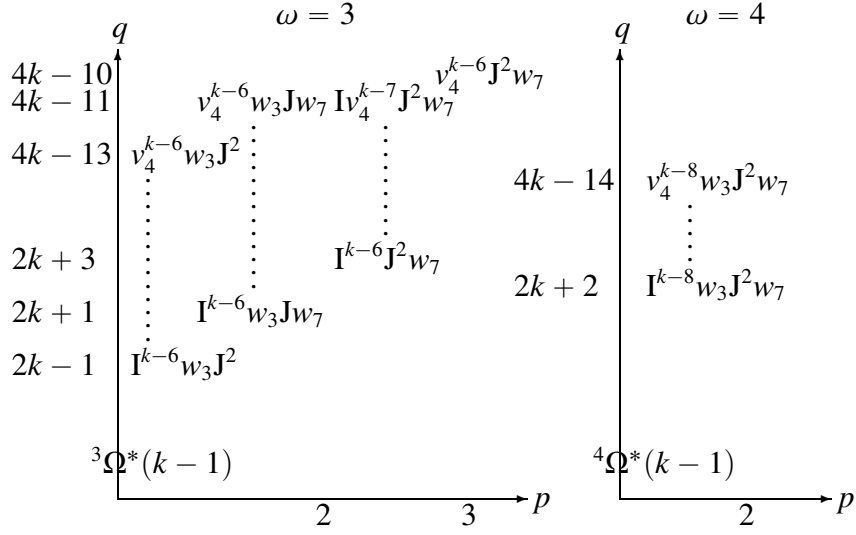
\begin{center}
\begin{picture}(360,200)
\put(110,190){$\omega=3$}            \put(265,190){$\omega=4$}
\put(50,10){\vector(0,1){170}}       \put(48,185){$q$}
\put(50,10){\vector(1,0){155}}       \put(207,7){$p$}
\put(240,10){\vector(0,1){170}}      \put(238,185){$q$}
\put(240,10){\vector(1,0){80}}       \put(322,7){$p$}
\put(40,20){${}^3\Omega^{*}(k-1)$}   \put(180,0){$3$}
\put(230,20){${}^4\Omega^{*}(k-1)$}          \put(10,57){$2k-1$}                  
\put(55,57){$\mbox{I}^{k-6}w_{3}\mbox{J}^2$} \put(10,77){$2k+1$}                  
\put(55,137){$v_{4}^{k-6}w_{3}\mbox{J}^2$}   \put(10,97){$2k+3$}                  
\put(80,77){$\mbox{I}^{k-6}w_3\mbox{J}w_7$}\put(10,137){$4k-13$}                
\put(80,157){$v_4^{k-6}w_3\mbox{J}w_7$}    \put(10,157){$4k-11$}                
\put(132,97){$\mbox{I}^{k-6}\mbox{J}^2w_7$}\put(10,167){$4k-10$}                
\put(132,157){$\mbox{I}v_4^{k-7}\mbox{J}^2w_7$}
\multiput(60,70)(0,5){13}{$\dot$}  
\multiput(100,90)(0,5){13}{$\dot$}
\multiput(150,110)(0,5){9}{$\dot$} 
\multiput(265,100)(0,5){5}{$\dot$}  
\put(170,167){$v_4^{k-6}\mbox{J}^2w_7$} \put(200,127){$4k-14$}  
\multiput(125,0)(160,0){2}{$2$}         \put(200,87){$2k+2$}                 
\put(250,87){$\mbox{I}^{k-8}w_3\mbox{J}^2w_7$}
\put(250,127){$v_4^{k-8}w_3\mbox{J}^2w_7$}
\end{picture}
\end{center}

The differential $d_{0}$ is $d_{0}(w_{3},w_{7})=(2v_{2}\bar{v}_{2},v_{4}^{2})$ and
$$ d_{0}(v_{2}^{\alpha}\bar{v}^{\,\beta}_{2}v_{4}^{\gamma}w_{5},v_{2}^{\alpha}\bar{v}^{\,\beta}_{2}v_{4}^{\gamma}\bar{w}_{5}) =\begin{cases}
      (0,0) & \mbox{if } \alpha=\beta=0\\
      (2v_{2}^{\alpha+1}\bar{v}_{2}^{\,\beta}v_{4}^{\gamma+1},2v_{2}^{\alpha}\bar{v}_{2}^{\,\beta+1}v_{4}^{\gamma+1}) & \mbox{if } \alpha+\beta\geq1.\\
   \end{cases}  $$
On the column $p=0$ we get ${}^{\omega}H^{*}(C_{k-1})$ and nothing on the columns $p=3$ and $p=4$: the differential $d_{0}$ is an isomorphism in the following case:
\begin{align*}
{}^{2}E_{0}^{3,4k-7}(k)=v_{4}^{k-4}\mbox{J}w_{7}&\xlongrightarrow{\cong}{}^{1}E_{0}^{3,4k-6}(k)=v_{4}^{k-2}\mbox{J}, \\
{}^{3}E_{0}^{3,4k-10}(k)=v_{4}^{k-6}\mbox{J}^{2}w_{7}&\xlongrightarrow{\cong}{}^{2}E_{0}^{3,4k-9}(k)=v_{4}^{k-4}\mbox{J}^{2},  \\
{}^{1}E_{0}^{4,4k-5}(k)=v_{4}^{k-2}w_{7}&\xlongrightarrow{\cong}{}^{2}E_{0}^{4,4k-4}(k)=v_{4}^{k}.
\end{align*}
On the column $p=2$ we have a five components cochain complex $e(q)$, where $q$ takes values in the interval $[k-2,2k-3]$:
\begin{center}
\begin{picture}(420,150)
\multiput(60,140)(90,0){4}{\vector(1,0){30}}
\put(70,145){$a$} \put(160,145){$b$}         \put(250,145){$c$}
\multiput(37,125)(90,0){5}{$\shortparallel$} \put(340,145){$d$}
\put(5,135){${}^{4}E_0^{2,2q-2}(k)$}         
\put(95,135){${}^{3}E_0^{2,2q-1}(k)$}
\put(190,135){${}^{2}E_0^{2,2q}(k)$}         
\put(275,135){${}^{1}E_0^{2,2q+1}(k)$}
\put(365,135){${}^{0}E_0^{2,2q+2}(k)$} \put(215,25){$\bigoplus$}
\put(0,110){$\mbox{I}^{2k-q-6}v_{4}^{q-k-2}w_{3}\mbox{J}^{2}w_{7}$}
\put(100,110){$\mbox{I}^{2k-q-6}v_{4}^{q-k}w_{3}\mbox{J}^{2}$}
\put(185,110){$\mbox{I}^{2k-q-5}v_{4}^{q-k+1}w_{3}\mbox{J}$}
\put(275,110){$\mbox{I}^{2k-q-4}v_{4}^{q-k+2}w_{3}$}
\put(355,110){$\mbox{I}^{2k-q-2}v_{4}^{q-k+2}$}
\multiput(125,95)(90,0){3}{$\bigoplus$}      
\multiput(125,60)(90,0){3}{$\bigoplus$}
\put(90,75){$\mbox{I}^{2k-q-5}v_{4}^{q-k-1}w_{3}\mbox{J}w_{7}$}
\put(195,75){$\mbox{I}^{2k-q-4}v_{4}^{q-k}\mbox{J}^{2}$}
\put(290,75){$\mbox{I}^{2k-q-3}v_{4}^{q-k+1}\mbox{J}$}
\put(90,40){$\mbox{I}^{2k-q-4}v_{4}^{q-k-2}\mbox{J}^{2}w_{7}$}
\put(195,40){$\mbox{I}^{2k-q-4}v_{4}^{q-k}w_{3}w_{7}$}
\put(290,40){$\mbox{I}^{2k-q-2}v_{4}^{q-k}w_{7}$}
\put(195,5){$\mbox{I}^{2k-q-3}v_{4}^{q-k-1}\mbox{J}w_{7}$}
\end{picture}
\end{center}
In the generic case, $q\in[k+2,2k-6]$, all the five components are non-zero and $e(q)$ is acyclic; the matrices of the differentials are
\begin{equation*}
a=
  \begin{pmatrix}
   \mbox{id}&*&*
  \end{pmatrix}
\quad
b=
  \begin{pmatrix}
      *   &  \mbox{id}  & 0\\
      *   & 0 &   \mbox{id}   \\
      0   &  * &   * \\
      0   &  *&   *
  \end{pmatrix}
\quad
c=
  \begin{pmatrix}
     * &   0   & \mbox{id}   & 0  \\
      *   & *  &   0   &  \mbox{id}  \\
      0     &  0 &   * &  *
  \end{pmatrix}
\quad
d=
  \begin{pmatrix}
    *&*&\mbox{id}
  \end{pmatrix}
  .
\end{equation*}
For the last values of $q$ the cochain complex $e(q)$ is shorter and still acyclic:\\
\begin{center}
\begin{picture}(360,150)
\multiput(43,140)(120,0){3}{\vector(1,0){60}}\put(-15,153){$q=2k-5$}
\multiput(12,125)(120,0){4}{$\shortparallel$}
\put(70,145){$b$} \put(187,145){$c$}         \put(312,145){$d$}
\put(-15,135){${}^{3}E_0^{2,4k-11}(k)$}     
\put(105,135){${}^{2}E_0^{2,4k-10}(k)$}
\put(230,135){${}^{1}E_0^{2,4k-9}(k)$}      
\put(345,135){${}^{0}E_0^{2,4k-8}(k)$}
\put(-10,110){$v_{4}^{k-6} w_{3}\mbox{J}w_{7}$}
\put(115,110){$v_{4}^{k-4}w_{3}\mbox{J} $}  
\put(240,110){$ \mbox{I}v_{4}^{k-3}w_{3}$} 
\put(362,110){$ \mbox{I}^{3}v_{4}^{k-3}$}   
\multiput(8,95)(120,0){3}{$\bigoplus$}
\put(-10,75){$\mbox{I}v_{4}^{k-7}\mbox{J}^{2}w_{7}$}
\put(117,75){$\mbox{I}v_{4}^{k-5}\mbox{J}^{2} $}
\put(240,75){$\mbox{I}^{2}v_{4}^{k-4}\mbox{J}$}
\multiput(128,60)(120,0){2}{$\bigoplus$} 
\put(128,25){$\bigoplus$}
\put(115,40){$\mbox{I}v_{4}^{k-5}w_{3}w_{7} $}
\put(240,40){$\mbox{I}^{3}v_{4}^{k-5} w_{7}$}
\put(115,5){$\mbox{I}^{2}v_{4}^{k-6}\mbox{J}w_{7} $}
\put(157,90){$\begin{pmatrix}
     * &   0   & \mbox{id}   & 0  \\
      *   & *  &   0   &  \mbox{id}  \\
      0     &  0 &   * &  *
  \end{pmatrix}$}
\put(50,90){$\begin{pmatrix}
       \mbox{id}  & 0 \\
       0 &   \mbox{id}\\
       * &   *        \\
       * &   *
  \end{pmatrix}$}
\put(295,95){$\begin{pmatrix}
    *&*&\mbox{id}
  \end{pmatrix}$}
\end{picture}
\end{center}

\begin{center}
\begin{picture}(360,110)
\multiput(107,100)(110,0){2}{\vector(1,0){50}} 
\put(179,20){$\bigoplus$}
\put(272,95){${}^{0}E_0^{2,4k-6}(k)$}  \put(127,105){$c$}
\put(237,105){$d$}                     \put(-15,113){$q=2k-4$}
\multiput(72,85)(110,0){3}{$\shortparallel$}   
\put(57,70){$v_{4}^{k-4}w_{3}w_{7}$}
\multiput(69,55)(110,0){2}{$\bigoplus$}        
\put(279,70){$\mbox{I}^{2}v_{4}^{k-2}$}
\put(57,35){$\mbox{I}v_{4}^{k-5}\mbox{J}w_{7}$}
\put(162,95){${}^{1}E_0^{2,4k-7}(k)$}
\put(167,35){$\mbox{I}v_{4}^{k-3}\mbox{J}$}    
\put(167,70){$v_{4}^{k-2}w_{3}$}
\put(167,0){$\mbox{J}^{2}v_{4}^{k-4}w_{7}$}    
\put(52,95){${}^{2}E_0^{2,4k-8}(k)$}
\put(110,65){$\begin{pmatrix}
       \mbox{id}  & 0 \\
       0 &   \mbox{id}\\
       * &   *
  \end{pmatrix}$}
\put(215,65){$\begin{pmatrix}
    *&*&\mbox{id}
  \end{pmatrix}$}
\end{picture}
\end{center}
\mbox{}
\begin{center}
\begin{picture}(360,60)
\put(95,40){${}^{1}E_0^{2,4k-5}(k)$} 
\put(150,45){\vector(1,0){50}}        
\put(205,40){${}^{0}E_0^{2,4k-4}(k)$}         \put(170,50){$d$}
\multiput(115,25)(110,0){2}{$\shortparallel$} 
\put(-15,58){$q=2k-3$}
\put(102,10){$\mbox{I}v_{4}^{k-3}w_{7}$}      
\put(215,10){$\mbox{I}v_{4}^{k-1}$}
\put(165,15){$\begin{pmatrix}
    \mbox{id}
  \end{pmatrix}$}
\end{picture}
\end{center}
For the first values of $q$ we obtain non-zero cohomology classes.
\begin{center}
\begin{picture}(360,60)
\put(150,40){\vector(1,0){50}}      \put(170,45){$d$}
\put(95,35){${}^{1}E_0^{2,2k-3}(k)$}          
\put(205,35){${}^{0}E_0^{2,2k-2}(k)$}
\multiput(115,25)(110,0){2}{$\shortparallel$} 
\put(-15,58){$q=k-2$}
\put(104,10){$\mbox{I}^{k-2}w_{3}$} \put(225,10){$\mbox{I}^{k}$}
\put(160,13){$\begin{pmatrix}
     0\\
    \mbox{id}\\
    0

  \end{pmatrix}$}
\end{picture}
\end{center}
and this gives ${}^{0}E_{1}^{2,2k-2}(k)=\langle v_{2}^{k},\bar{v}_{2}^{k}\rangle$.
\begin{center}
\begin{picture}(360,80)
\multiput(120,65)(110,0){2}{\line(1,0){48}}
\multiput(113,61.7)(150,0){2}{$>$}  
\multiput(160,61.7)(110,0){2}{$>$}
\put(60,60){${}^{2}E_0^{2,2k-2}(k)$}   \put(-15,78){$q=k-1$}
\put(170,60){${}^{1}E_0^{2,2k-1}(k)$}        
\put(280,60){${}^{0}E_0^{2,2k}(k)$}
\multiput(80,50)(110,0){3}{$\shortparallel$}\put(135,70){$c$}
\put(65,35){$\mbox{I}^{k-4}w_3\mbox{J}$}     
\put(175,35){$\mbox{I}^{k-3}v_4w_3$}
\put(290,35){$\mbox{I}^{k-1}v_{4}$}    \put(187,20){$\bigoplus$}
\put(180,0){$\mbox{I}^{k-2}\mbox{J}$}
\end{picture}
\end{center}

Obviously, the first differential is injective and the second is surjective, the Euler characteristic is $(2k-6)-(3k-4)+k=-2$ and 
$\langle v_{2}^{k-3}\gamma,\bar{v}_{2}^{k-3}\bar{\gamma}\rangle$ is a complement for the image of $c$, hence ${}^{1}E_{1}^{2,2k-1}(k)=\langle v_{2}^{k-3}\gamma,\bar{v}_{2}^{k-3}\bar{\gamma}\rangle$.
\begin{center}
\begin{picture}(360,120)
\multiput(89,101.7)(120,0){3}{$>$}  \multiput(36,101.7)(285,0){2}{$>$}
\multiput(40,105)(120,0){3}{\line(1,0){54}} \put(307,110){$d$}
\multiput(7,85)(120,0){4}{$\shortparallel$} \put(182,110){$c$}
\put(-20,100){${}^{3}E_0^{2,2k-1}(k)$}        
\put(100,100){${}^{2}E_0^{2,2k}(k)$}
\put(225,100){${}^{1}E_0^{2,2k+1}(k)$}        
\put(340,100){${}^{0}E_0^{2,2k+2}(k)$}
\put(-15,70){$\mbox{I}^{k-6} w_{3}\mbox{J}^{2}$}
\put(110,70){$\mbox{I}^{k-5}v_{4}w_{3}\mbox{J}$}
\put(235,70){$ \mbox{I}^{k-4}v_{4}^{2}w_{3}$} 
\put(357,70){$ \mbox{I}^{k-2}v_{4}^{2}$}
\multiput(123,55)(120,0){2}{$\bigoplus$}    \put(65,110){$b$}
\put(115,35){$\mbox{I}^{k-4}\mbox{J}^{2} $}   
\put(235,0){$\mbox{I}^{k-2} w_{7}$}
\put(235,35){$\mbox{I}^{k-3}v_{4}\mbox{J}$} \put(-20,118){$q=k$}
\multiput(123,20)(120,0){2}{$\bigoplus$}      
\put(110,0){$\mbox{I}^{k-4}w_{3}w_{7}$}
\end{picture}
\end{center}

Definitely, $\langle v_{2}^{k-3}\varepsilon,\bar{v}_{2}^{k-3}\bar{\varepsilon}\rangle\subset\mbox{ker}(d)$ and its intersection with $\mbox{Im}(e)$  is 0. The subcomplex $f(k)\subset e(k)$ generated by $v_{4}^{2}$ and $w_{7}$ is acyclic 
($\alpha v_{4}^{2}\mapsto \alpha w_{7}$ gives a homotopy $\mbox{id}_{f(k)}\simeq0$) and the quotient complex $e(k)/f(k)$ is
$$ \mbox{I}^{k-6}w_{3}\mbox{J}^{2}\rightarrowtail \mbox{I}^{k-5}v_{4}w_{3}\mbox{J}\oplus\mbox{I}^{k-4}\mbox{J}^{2}\rightarrow\mbox{I}^{k-3}v_{4}\mbox{J} $$
with dimensions $k-5$, $3k-11$ and $2k-4$ respectively. Therefore ${}^{1}E_{1}^{2,2k+1}(k)$ has dimension 2 and it is equal to $\langle v_{2}^{\,k-3}\varepsilon,\bar{v}_{2}^{\,k-3}\bar{\varepsilon}\rangle$.
\begin{center}
\begin{picture}(370,160)
\multiput(95,136.7)(120,0){3}{$>$} 
\multiput(41,136.7)(285,0){2}{$>$}
\multiput(46,140)(120,0){3}{\line(1,0){53}}  
\put(-10,150){$q=k+1$}
\multiput(12,125)(120,0){4}{$\shortparallel$}
\put(128,25){$\bigoplus$}
\put(-15,135){${}^{3}E_0^{2,2k+1}(k)$}       
\put(105,135){${}^{2}E_0^{2,2k+2}(k)$}
\put(230,135){${}^{1}E_0^{2,2k+3}(k)$} 
\put(345,135){${}^{0}E_0^{2,4k+4}(k)$}
\put(-15,110){$\mbox{I}^{k-7}v_{4} w_{3}\mbox{J}^{2}$}
\put(115,110){$\mbox{I}^{k-6}v_{4}^{2}w_{3}\mbox{J} $}
\put(240,110){$ \mbox{I}^{k-5}v_{4}^{3}w_{3}$}
\put(362,110){$ \mbox{I}^{k-3}v_{4}^{3}$}    
\multiput(8,95)(120,0){3}{$\bigoplus$}
\put(-10,75){$\mbox{I}^{k-6}w_{3}\mbox{J}w_{7}$}
\put(117,75){$\mbox{I}^{k-5}v_{4}\mbox{J}^{2}$}
\put(240,75){$\mbox{I}^{k-4}v_{4}^{2}\mbox{J}$}
\multiput(128,60)(120,0){2}{$\bigoplus$}     
\put(240,40){$\mbox{I}^{k-3}v_{4} w_{7}$}
\put(115,40){$\mbox{I}^{k-5}v_{4}w_{3}w_{7} $}
\put(115,5){$\mbox{I}^{k-4}\mbox{J}w_{7} $}
\end{picture}
\end{center}

As in the previous case, $\langle v_{2}^{\,k-5}\eta,\bar{v}_{2}^{\,k-5}\bar{\eta}\rangle\subset\mbox{ker}(e)$ and its intersection with $\mbox{Im}(b)$ is 0. The same subcomplex $f(k+1)\subset e(k+1)$ is acyclic and in the quotient subcomplex
$$ \mbox{I}^{k-3}v_{4}w_{3}\mbox{J}^{2}\rightarrowtail\mbox{I}^{k-5}v_{4}\mbox{J}^{2} $$
the dimensions are $k-6$ and $k-4$. Hence ${}^{2}E_{1}^{2,2k+2}(k)=\langle v_{2}^{\,k-5}\eta,\bar{v}_{2}^{\,k-5}\bar{\eta}\rangle.$

In conclusion, the spectral sequences $\{{}^{*}E_{*}^{*,*}(k)\}_{k\geq8}$ degenerate at ${}^{*}E_{1}^{*,*}$ with the described eight cohomology classes in ${}^{*}E_{*}^{\geq1,*}(k)$.
\end{proof}
As a consequence of the computation we have the following table of the two-variables Poincar\'{e} polynomials (for each $k$, the first line contains the coefficients corresponding to $s=0$, the second line those with $s=1$ and the third line corresponds to $s=2$) and the proofs of Corollaries \ref{cor.2}, \ref{cor.3} and Proposition \ref{prop.5}.
\begin{center}
Table 4
\vspace{1pt}\\
\begin{tabular}{ |c|cccccccccccccccc|  }
\hline
$k$     &0 &  2&  4&  6 &7& 8& 9& 10&11& 12& 13& 14 &15 & 16& 17& 18\\
\hline
1       & 1 &  2&  1&   & & & & & & & &  & & & & \\
\hline
2       & 1 &  2&  3&   & & & & & & & &  & & & & \\
\hline
3       & 1 &  2& 3  & 2 & & & & & & & &  & & & & \\
        &     & & &  & 2& & 2& & & & &  & & & & \\
\hline
4       & 1 &2  & 3&   2& &2 & & & & & &  & & & & \\
        &   & & &  & 2& & 4& & 3& & &  & & & & \\
\hline
5       & 1 &  2&  3&  2 & &2 & &2 & & & &  & & & & \\
        &  &  & &  & 2& & 4& &5 & & 2&  & & & & \\
        &   & & &  & & & & & & & & 1 & & & & \\
\hline
6       & 1 &  2&  3&  2 & & 2& &2 & & 2& &  & & & & \\
        &   & & &  &2 & &4 & & 5& & 4&  &2 & & & \\
        &   & & &  & & & & & & & & 1 & &2 & & \\
\hline
7       & 1 &  2&  3&  2 & & 2& & 2& & 2& & 2 & & & & \\
        &  & & &  &2 & &4 & & 5& &4 &  & 4& & 2& \\
        &   & & &  & & & & & & & &  1& & 2& & 2\\
\hline
\end{tabular}
\end{center}
\begin{corollary}\label{cor.2}
The space $\mathbb{C}\emph{P}^{1}\times\mathbb{C}\emph{P}^{1}$ satisfies the Poincar\'{e} polynomial shifted stability condition with range 6, shift 2 and recurrence relation
$$ P_{C_{k+1}(\mathbb{C}\emph{P}^{1}\times\mathbb{C}\emph{P}^{1})}(t,s)=P_{C_{k}(\mathbb{C}\emph{P}^{1}\times\mathbb{C}\emph{P}^{1})}(t,s)+2t^{2k+2}[1+s(t+t^{2})+s^{2}t^{4}]\quad(k\geq6). $$
\end{corollary}
\begin{corollary}\label{cor.3}
The space $\mathbb{C}\emph{P}^{1}\times\mathbb{C}\emph{P}^{1}$ satisfies the extended shifted stability condition with range 6 and shift 2. For any $k\geq 6$ we have:
$$ P_{C_{k+1}(\mathbb{C}\emph{P}^{1}\times\mathbb{C}\emph{P}^{1})}^{[2(k-5)]}(t,s)=t^{2}P_{C_{k}(\mathbb{C}\emph{P}^{1}\times\mathbb{C}\emph{P}^{1})}^{[2(k-5)]}(t,s). $$
\end{corollary}
{\em Proof of Proposition \ref{prop.5}}. Obvious from Corollaries \ref{cor.2} and \ref{cor.3}. $\hfill \square$

With a different terminology, that of ``stable instability,'' M. Maguire proved in \cite{M} the shifted stability property for the complex projective space $\mathbb{C}\mbox{P}^{3}.$ Using our method one can obtain M. Maguire's result as Proposition \ref{prop.7} and Corollaries \ref{cor.4}, \ref{cor.5}
\begin{prop}\label{prop.7}
The complex projective space $\mathbb{C}\mbox{P}^{3}$ satisfies the spectral sequence shifted stability condition with range 6 and shift 2. More precisely, the non-zero pieces ${}^{*}E_{\infty}^{\geq1,*}(k)$ (for $k\geq6$) are:
$$ {}^{0}E_{\infty}^{2,2k-2}(k),\quad{}^{1}E_{\infty}^{2,2k+3}(k),\quad{}^{1}E_{\infty}^{2,2k+5}(k)\quad\mbox{and}\quad{}^{2}E_{\infty}^{2,2k+10}(k) $$
and all these spaces have dimension one.
\end{prop}
The next table contains the coefficients of the two-variables Poincar\'{e} polynomials of the first configuration spaces $C_{k}(\mathbb{C}\mbox{P}^{3})$ (we use the convention of Table 4):
\begin{center}
Table 5
\vspace{1pt}\\
\begin{tabular}{ |c|cccccccccccccccc| }
\hline
$k$     & 0 & 2&  4&  6 & 8&  10&11& 12& 13& 14 &15 &  17& 19&21&24&26\\
\hline
1       & 1 & 1 & 1  & 1 &   & & & & &  & & & & & &  \\
\hline
2       & 1 & 1 & 2 & 1   &1 &  & & & &  & & & & & & \\
\hline
3       & 1 & 1  &2  & 2  &1 &  & & & &  & & & & & &  \\
        &   & & &  & &  &1 & & 1&  &1 & & & & & \\
\hline
4       & 1& 1  &2  &2   &2  & & & & &  & & & & & & \\
        & &  & &  & &  &1 & & 2&  &2  &1 &1 & & &\\
\hline
5       &1 & 1  &2  &2   & 2&  1& & & &  & & &  & & &\\
        & &  & &  & &  &1 & &2 &  & 3&2 &1 & & &\\
\hline
6       & 1& 1 & 2  &2   &2  &1 & &1 & &  & & & & & & \\
        & &  & &  & & &  1& & 2&  &3  &3 &2 & & &\\
        & &  & &  & & & & & &  & & &  & & 1&\\
\hline
7       & 1&1  &2  &2   &2  &1 & & 1& & 1 & & & & & & \\
        & &  & &  & &  &1 & &2 &  & 3&  3&3 & 1& &\\
        & &  & &  & & & & & &  & & &  & & 1&1\\
\hline
\end{tabular}
\end{center}

\begin{corollary}\label{cor.4}
The space $\mathbb{C}\mbox{P}^{3}$ satisfies the Poincar\'{e} polynomial shifted stability condition with range 5, shift 2 and recurrence relation
$$P_{C_{k+1}(\mathbb{C}\emph{P}^{3})}(t,s)=P_{C_{k}(\mathbb{C}\emph{P}^{3})}(t,s)+t^{2k+2}[1+s(t^{5}+t^{7})+s^{2}t^{12}]\quad(k\geq5).$$
\end{corollary}
\begin{corollary}\label{cor.5}
The spaces $\mathbb{C}\mbox{P}^{3}$ satisfies the extended shifted stability condition with range 6 and shift 2. For any $k\geq6$ we have:
$$P_{C_{k+1}(\mathbb{C}\emph{P}^{3})}^{[2(k-5)]}(t,s)=t^{2}P_{C_{k}(\mathbb{C}\emph{P}^{3})}^{[2(k-5)]}(t,s).$$
\end{corollary}
The complete details of the proofs of these and other results for the unordered configuration spaces of $\mathbb{C}\mbox{P}^{n}$ will be given in \cite{BY}.\\
{\em Proof of Proposition \ref{prop.6}}. Obvious from Corollaries \ref{cor.4} and \ref{cor.5}. $\hfill \square$
\begin{remark}
In these two examples, $\mathbb{C}\mbox{P}^{1}\times\mathbb{C}\mbox{P}^{1}$ and $\mathbb{C}\mbox{P}^{3},$ the sequence of odd Betti numbers is unimodal for any $k.$ This is not true for the sequence of even Betti numbers, but the sequences of  Betti numbers $\{\beta_{i,j}\}_{i\geq 1}$ are unimodal too, for each $j\in\{0,1,2\}$.
\end{remark}

\newcommand{\arxiv}[1]
{\texttt{\href{http://arxiv.org/abs/#1}{arxiv:#1}}}
\newcommand{\arx}[1]
{\texttt{\href{http://arxiv.org/abs/#1}{arXiv:}}
\texttt{\href{http://arxiv.org/abs/#1}{#1}}}
\newcommand{\doi}[1]
{\texttt{\href{http://dx.doi.org/#1}{doi:#1}}}

\end{document}